\documentclass[12pt]{amsart}
\usepackage{amsmath,amsthm,amssymb,amscd,url,enumerate,mathtools}
\usepackage{graphicx}
\usepackage[margin=1in]{geometry}
\usepackage{afterpage}
\usepackage{tikz}
\usepackage[loadshadowlibrary]{todonotes}
\usepackage[all]{xy}
\usepackage[colorlinks=true,citecolor={blue},linkcolor = {purple},
pdfauthor={Dylan Scofield and Hanson Smith}, pdftitle={Prime Splitting and Common N-Index Divisors in Radical Extensions}, pdfsubject={Algebraic Number Theory},
            pdfkeywords={Radical extension, Pure extension, Prime splitting, Prime ideal factorization, Monogenic}]{hyperref} 
\usepackage{tabularx}
\usepackage{subcaption} 
\usepackage{float}		
\usepackage{booktabs}
\usepackage{subcaption}
\usepackage{comment}
\usepackage{caption}

\usepackage{enumitem} 
\usepackage{refcount} 
\usepackage{etoolbox} 
\usepackage{multirow}

\usepackage{multicol}
\usepackage{relsize} 

\usepackage{dutchcal}

\usepackage[square]{natbib}
\setcitestyle{numbers}

\newcommand{\TITLE}{Prime Splitting and Common $N$-Index Divisors in Radical Extensions:\\  Part $p=2$}
\newcommand{\TITLERUNNING}{Splitting in Radical Extensions}

\theoremstyle{plain}
\newtheorem{theorem}{Theorem}

\newtheorem{proposition}[theorem]{Proposition}
\newtheorem{lemma}[theorem]{Lemma}
\newtheorem{corollary}[theorem]{Corollary}

\theoremstyle{definition}

\theoremstyle{remark}
\newtheorem{remark}[theorem]{Remark}

\newtheorem{example}[theorem]{Example}

\numberwithin{theorem}{section}

  {\end{list}}

  {\end{list}}

\newcommand{\tightoverset}[2]{%
  \mathop{#2}\limits^{\vbox to -.5ex{\kern-1.05ex\hbox{$#1$}\vss}}}

\numberwithin{equation}{section} 

\newcommand{\gm}{{\mathfrak{m}}}
\newcommand{\gn}{{\mathfrak{n}}}
\newcommand{\gp}{{\mathfrak{p}}}

\newcommand{\gP}{{\mathfrak{P}}}

\def\Bcal{{\mathcal B}}

\def\Ocal{{\mathcal O}}

\newcommand{\FF}{\mathbb{F}}

\newcommand{\QQ}{\mathbb{Q}}

\newcommand{\ZZ}{\mathbb{Z}}

\newcommand{\dnd}{\nmid}

\newcommand{\ol}[1]{\overline{#1}}

\newcommand{\Disc}{\operatorname{Disc}}

\newcommand{\Irr}{\operatorname{Irred}}
\newcommand{\red}{\operatorname{red}}

\title[\TITLERUNNING]{\TITLE}

\author[Dylan Scofield]{Dylan Scofield}
\address{Department of Mathematics, California State University San Marcos,
333 S. Twin Oaks Valley Rd.
San Marcos, CA 92096
USA}
\email{scofi004@csusm.edu}

\author[Hanson Smith]{Hanson Smith}
\address{Department of Mathematics, California State University San Marcos,
333 S. Twin Oaks Valley Rd.
San Marcos, CA 92096
USA}
\email{hsmith@csusm.edu}

\keywords{Radical extension, Pure extension, Prime splitting, Prime ideal factorization} 
\subjclass[2020]{11R04, 11R21, 11R27}

\begin{document}

\sloppy 


\baselineskip=17pt


\begin{abstract}
Following work of V\'elez, we explicitly describe the splitting of the integral prime 2 in the radical extension $\mathbb{Q}(\sqrt[n]{a})$, where $x^n-a$ is an irreducible polynomial in $\mathbb{Z}[x]$. With previous work of the second author, this fully describes the splitting of any prime in $\mathbb{Q}(\sqrt[n]{a})$. Using this description, we classify common index divisors (the primes whose splitting prevents the existence of a power integral basis for the ring of integers). Using work of Pleasants, we extend this to describe common $N$-index divisors (primes that divide the index of any order generated over $\mathbb{Z}$ by $N$ elements). We also present a novel construction of non-monogenic fields with no common index divisors as well as constructions of number rings requiring $N+1$ ring generators for any $N>0$. Examples are provided throughout.
\end{abstract}


\maketitle




\section{Introduction and Main Theorem}

Let $x^n-a$ be an irreducible polynomial in $\ZZ[x]$. The primary objective of this paper is to establish Theorem \ref{Thm: Main}, an explicit description of the splitting of the integral prime 2 in $\Ocal_{\QQ(\hspace{-.5ex}\sqrt[n]{a})}$ the ring of integers of $\QQ\big(\hspace{-.5ex}\sqrt[n]{a}\big)$. Combining this with Theorem \ref{Thm: MainOddp}, the result for odd primes proved in \cite{SmithRadicalSplitting}, we achieve a complete, explicit description of prime splitting in radical\footnote{These are also known as \textit{pure extensions} and \textit{root extensions}.} extensions of $\QQ$. Numerous examples of these theorems are computed in Section \ref{Sec: Examples}, and a SageMath implementation can be found here: \href{https://doi.org/10.5281/zenodo.20208664}{https://doi.org/10.5281/zenodo.20208664}.

In Corollaries \ref{Cor: 2NCID} and \ref{Cor: oddNCIDs}, this description is employed to classify the common $N$-index divisors (primes dividing the index of any order $\ZZ\big[\theta_1,\dots, \theta_N\big]$ in the maximal order $\Ocal_K$) of a radical extension. Using this we can construct radical extensions with common $N$-index divisors for arbitrary $N$; see Example \ref{Ex: Using2togetNgenerators}. In Proposition \ref{Prop: NonRadCommonN}, we construct non-radical extensions of minimal degree having a common $N$-index divisor. In the other direction, in Section \ref{Sec: NonMonoNoCIDs} we use the index form to construct extensions that are not monogenic yet have no common index divisors. 

Before we state our theorems, we make the following definition. Let $L/K$ be an extension of number fields with respective rings of integers $\Ocal_L$ and $\Ocal_K$. Let $\gp\subset \Ocal_K$ be a prime ideal with residue field $k_\gp\coloneqq \Ocal_K/\gp$. An ideal $\mathfrak{I}\subset \Ocal_L$ \textit{mirrors} the factorization of a polynomial $f(x)\in k_\gp[x]$ if, given a factorization into irreducibles $f(x)=\prod_{i=1}^r \phi_i(x)^{e_i}$ in $k_\gp[x]$, there is a prime ideal factorization $\mathfrak{I}=\prod_{i=1}^r\gP_i^{e_i}$ in $\Ocal_L$ where the degree of $\gP_i$ over $\gP_i\cap \Ocal_K$, i.e. $[\Ocal_L/\gP_i:k_\gp]$, is equal to the degree of $\phi_i(x)$. 

We can now state our main theorem.

\begin{theorem}[The prime ideal factorization of $2\Ocal_{\QQ(\hspace{-.5ex}\sqrt[n]{a})}$]\label{Thm: Main}
    Suppose $x^n-a\in \ZZ[x]$ is irreducible. Changing variables, we assume that $a$ is $n^{\text{th}}$ power free. Let $2\Ocal_{\QQ(\hspace{-.5ex}\sqrt[n]{a})}$ denote the ideal generated by $2$ in the ring of integers of $\QQ\big(\hspace{-.5ex}\sqrt[n]{a}\big)$. The ideal factors of $2\Ocal_{\QQ(\hspace{-.5ex}\sqrt[n]{a})}$ mirror the factorization of a monic, separable polynomial $g(x)~\in~\ZZ[x]$ over either $\FF_2[x]$ or $\FF_4[x]$:
    \begin{equation}
        \ol{g(x)}=\prod_{i=1}^r \gamma_i(x) \ \text{ over } \ \FF_2[x] \ \text{ and } \ \ol{g(x)}=\prod_{j=1}^s \Gamma_j(x) \ \text{ over } \ \FF_4[x]. 
    \end{equation}
    Using indices to indicate the relevant factorization, $1\leq i\leq r$ corresponds to a factorization over $\FF_2[x]$ with $\deg\gamma_i(x)=[\Ocal_{\QQ(\hspace{-.5ex}\sqrt[n]{a})}/\gp_i:\FF_2]$, while $1\leq j\leq s$ corresponds to a factorization over $\FF_4[x]$ with $\deg\Gamma_j(x)=[\Ocal_{\QQ(\hspace{-.5ex}\sqrt[n]{a})}/\gp_j:\FF_4]$. In particular, we have the following prime ideal factorizations:

    \begin{enumerate}[label=\Roman*., ref=\Roman*]
        \item If $na$ is odd, then $g(x)=x^n-a$ and $\displaystyle 2 \Ocal_{\QQ(\hspace{-.5ex}\sqrt[n]{a})} =\prod_{i=1}^r\gp_i$, with $\gp_i=\big(2,\gamma_i(\hspace{-.5ex}\sqrt[n]{a})\big)$.\label{MainI}

        \item  Otherwise, if $2\mid a$ and $2\nmid\gcd\big(v_2(a),n\big)$, then $g(x)=x^{\gcd(v_2(a),n)}-1$ 
      and \[\displaystyle 2 \Ocal_{\QQ(\hspace{-.5ex}\sqrt[n]{a})} =\prod_{i=1}^r\gp_i^{n / \gcd(v_2(a),n)}.\]\label{MainII}

        \item Otherwise, if $2\mid n$ and $2\nmid a$, then we write $n=2^m n_0$ with $n_0$ odd and let $w=v_2(a-1)$. Here $g(x)=x^{n_0}-1$. \label{MainIII}
        \begin{enumerate}[label=\roman*., ref=III.\roman*]
            \item If $w=1$, then $\displaystyle 2 \Ocal_{\QQ(\hspace{-.5ex}\sqrt[n]{a})} =\prod_{i=1}^r\gp_i^{2^m}$.\label{MainIIIi}

            \item If $2\leq w\leq m+1$, then $\displaystyle 2 \Ocal_{\QQ(\hspace{-.5ex}\sqrt[n]{a})} =\prod_{j=1}^s\gp_j^{2^{m-w+1}} \prod_{l=2}^{w-1} \prod_{i=1}^r\gp_{l,i}^{2^{m-w+l}} .$\label{MainIIIii}
            
            \vspace{.05 in}
            
            \item If $w>m+1$, then $\displaystyle 2 \Ocal_{\QQ(\hspace{-.5ex}\sqrt[n]{a})} =\prod_{l=1}^{m+1} \prod_{i=1}^r\gp_{l,i}^{2^{\max(l-2,0)}} .$ \label{MainIIIiii}
        \end{enumerate}

\vspace{.05 in}

        \item Otherwise, $2\mid a$ and $2\mid\gcd\big(v_2(a),n\big)$. We write $a=a_02^{h2^k}$ with $a_0$ and $h$ odd and $k\geq 1$. 
        We let $w_0=v_2(a_0-1)$ and $t=\max(0,w_0-2)$. With the superscript indicating perfect powers, $a_0\in \ZZ_2^{2^{t}}$ but $a_0\notin \ZZ_2^{2^{t+1}}$. Recalling $n=2^m n_0$ with $n_0$ odd, we let $d=\gcd(h,n_0)$. 
        Here we have $g(x)=x^d-1$. \label{MainIV}
        \begin{enumerate}[label= Case \O., ref=IV.\O]
            \item If $k\geq m$, then $\QQ\big(\hspace{-.5ex}\sqrt[2^m]{a}\big)=\QQ\big(\hspace{-.5ex}\sqrt[2^m]{a_0}\big)$, so it is similar to Case \ref{MainIII} with $w=w_0$. \label{MainIV0}
                \begin{enumerate}[label=\arabic*., ref=IV.\O.\arabic*]
                    \item If $w_0=1$, then $\displaystyle 2 \Ocal_{\QQ(\hspace{-.5ex}\sqrt[n]{a})} =\prod_{i=1}^r\gp_i^{\frac{n_0}{d} \cdot 2^m}$.\label{MainIV0i}

                    \item If $2\leq w_0\leq m+1$, then $\displaystyle 2 \Ocal_{\QQ(\hspace{-.5ex}\sqrt[n]{a})} =\prod_{j=1}^s\gp_j^{\frac{n_0}{d} \cdot 2^{m-w_0+1}} \prod_{l=2}^{w_0-1} \prod_{i=1}^r\gp_{l,i}^{\frac{n_0}{d} \cdot 2^{m-w_0+l}} .$\label{MainIV0ii}
            
                    \vspace{.05 in}
            
                    \item If $w_0>m+1$, then $\displaystyle 2 \Ocal_{\QQ(\hspace{-.5ex}\sqrt[n]{a})} =\prod_{l=1}^{m+1} \prod_{i=1}^r\gp_{l,i}^{\frac{n_0}{d} \cdot 2^{\max(l-2,0)}} .$ \label{MainIV0iii}
            \end{enumerate}
        \end{enumerate}
        For the remaining subcases, $k<m$.
        \begin{enumerate}[label=\roman*., ref=IV.\roman*]        
            \item If $w_0=1$, then there are a number of cases. Note $v_2(a_0+1)>1$, since $w_0=1$. \label{MainIVi}
                \begin{enumerate}[label=\arabic*., ref=IV.i.\arabic*]
                    \item If $v_2(a_0+1)=2$ or if $v_2(a_0+1)\geq 3$ and $k\geq 2$, then $\displaystyle 2 \Ocal_{\QQ(\hspace{-.5ex}\sqrt[n]{a})} =\prod_{i=1}^r \gp_i^{\frac{n_0}{d}\cdot 2^{m}}$. \label{MainIVi1}

\vspace{.05 in}
                    
                    \item If $v_2(a_0+1)=3$ and $k=1$, then $\displaystyle 2 \Ocal_{\QQ(\hspace{-.5ex}\sqrt[n]{a})} =\prod_{j=1}^s \gp_j^{\frac{n_0}{d}\cdot2^{m-1}}$. \label{MainIVi2} 

\vspace{.05 in}
                    
                    \item If $v_2(a_0+1)\geq 4$ and $k=1$, then $\displaystyle 2 \Ocal_{\QQ(\hspace{-.5ex}\sqrt[n]{a})} =\prod_{i=1}^r \gp_{1,i}^{\frac{n_0}{d}\cdot 2^{m-1}}\gp_{2,i}^{\frac{n_0}{d}\cdot 2^{m-1}}$. \label{MainIVi3}
                \end{enumerate}

            \item If $w_0=2$, then $\displaystyle 2 \Ocal_{\QQ(\hspace{-.5ex}\sqrt[n]{a})} = \prod_{j=1}^s \gp_j^{\frac{n_0}{d}\cdot2^{m-1}}.$ \label{MainIVii}

            \item If $w_0\geq 3$, then $t=w_0-2$ and we again have a number of cases: \label{MainIViii}
                \begin{enumerate}[label=\arabic*., ref=IV.iii.\arabic*]
                
                    \item If $t < k$, then $\displaystyle 2 \Ocal_{\QQ(\hspace{-.5ex}\sqrt[n]{a})} = \prod_{j=1}^s\gp_{0,j}^{\frac{n_0}{d}\cdot 2^{m-t-1}}  \prod_{l=1}^t  \prod_{i=1}^r\gp_{l,i}^{\frac{n_0}{d}\cdot 2^{m-t+l-1}} .$ \label{MainIViii1}

\vspace{.05 in}

                    \item If $t\geq k=1$, then $\displaystyle 2 \Ocal_{\QQ(\hspace{-.5ex}\sqrt[n]{a})} = \prod_{i=1}^r\left( \gp_{0,i}\gp_{1,i}\right)^{\frac{n_0}{d}\cdot 2^{m-1}}$. \label{MainIViiinew2}
                    
                    \item If $t=k>1$, then $\displaystyle 2 \Ocal_{\QQ(\hspace{-.5ex}\sqrt[n]{a})} = \prod_{j=1}^s\gp_{2,j}^{\frac{n_0}{d}\cdot 2^{m-t}} \hspace{-.1em} \prod_{i=1}^r \hspace{-.2em}\left(\gp_{0,i}\gp_{1,i}\gp_{3,i}^{2^2}\gp_{4,i}^{2^3}\cdots \gp_{t,i}^{2^{t-1}}\right)^{\frac{n_0}{d}\cdot 2^{m-t}}$, where the product $\prod_{l=3}^t\gp_{l,i}^{2^{l-1}}$ is empty if $t=2$. \label{MainIViiinew3}
\vspace{.05 in}
                    
                    \item If $t>k>1$, then $\displaystyle 2 \Ocal_{\QQ(\hspace{-.5ex}\sqrt[n]{a})} = \prod_{i=1}^r\left( \gp_{0,i}\gp_{1,i}\gp_{2,i}\gp_{2,i}'\gp_{3,i}^{2^2}\cdots \gp_{k,i}^{2^{k-1}} \right)^{\frac{n_0}{d}\cdot 2^{m-k}},$\\
                    where the product $\prod_{l=3}^k\gp_{l,i}^{2^{l-1}}$ is empty if $k=2$. \label{MainIViiinew4} 
                \end{enumerate}
        \end{enumerate}
    \end{enumerate}
\end{theorem}


The complement to Theorem \ref{Thm: Main} for odd $p$ is \cite[Theorem 1.2]{SmithRadicalSplitting}. We restate it here.

\begin{theorem}[The prime ideal factorization of $p\Ocal_{\QQ(\hspace{-.5ex}\sqrt[n]{a})}$]\label{Thm: MainOddp}
Suppose $x^n-a\in \ZZ[x]$ is irreducible with $a$ $n^\text{th}$ power free and let $p$ be an odd integral prime. The ideal factors of $p\Ocal_{\QQ(\hspace{-.5ex}\sqrt[n]{a})}$ mirror the factorization of a monic polynomial $g(x)\in\ZZ[x]$ over $\FF_p[x]$: $\displaystyle \ol{g(x)}=\prod_{i=1}^r \gamma_i(x).$ We use indices to indicate the residue class degree of $\gp_i$ over $p$ corresponds to the degree of $\gamma_i(x)$. In particular, we have the following prime ideal factorizations:

\begin{enumerate}[label=\Roman*., ref=\Roman*]
        \item If $p\dnd na$, then $g(x)=x^n-a$ and $\displaystyle p \Ocal_{\QQ(\hspace{-.5ex}\sqrt[n]{a})} =\prod_{i=1}^r\gp_i$. \label{OddMainI}

        \item Otherwise, if $p\mid a$ and $p\nmid\gcd\big(v_p(a),n\big)$, then, writing $a=a_0p^{hp^k}$ with $p\dnd a_0$, 
        we have $g(x)=x^{\gcd(h,n)}-a_0$ and $\displaystyle p \Ocal_{\QQ(\hspace{-.5ex}\sqrt[n]{a})} =\prod_{i=1}^r\gp_i^{n / \gcd(h,n)}.$ \label{OddMainII}

        \item Otherwise, if $p\mid n$ and $p\nmid a$, then we write $n=p^m n_0$ with $p\dnd n_0$ and we let $w=v_p\big(a^{p-1}-1\big)$. 
        Here $g(x)=x^{n_0}-a$ and $\varphi$ is Euler's totient function. \label{OddMainIII}
        \begin{enumerate}[label=\roman*., ref=III.\roman*]
            \item If $1\leq w \leq m+1$, then $\displaystyle p \Ocal_{\QQ(\hspace{-.5ex}\sqrt[n]{a})} =\prod_{l=1}^w\prod_{i=1}^r\gp_{l,i}^{\varphi(p^{l-1})p^{m-w+1}}$. \label{OddMainIIIi}

            
            \vspace{.05 in}
            
            \item If $w>m+1$, then $\displaystyle  p \Ocal_{\QQ(\hspace{-.5ex}\sqrt[n]{a})} = \prod_{l=1}^{m+1} \prod_{i=1}^r \gp_{l,i}^{\varphi(p^{l-1})} .$ \label{OddMainIIIii}
        \end{enumerate}

\vspace{.05 in}

        \item Otherwise, $p\mid a$ and $p\mid\gcd\big(v_p(a),n\big)$. Recalling $a=a_0p^{hp^k}$ with $p\dnd a_0$, we let $w_0=v_p\big(a_0^{p-1}-1\big)$ and $c=\min(w_0-1,k,m)$. Here we have two polynomials $g(x)$ to consider. The $\gp_{0,*}$ correspond to factors of $\displaystyle g_0(x)=x^{\gcd(n_0,h)}-a_0 \equiv \prod_{\mathfrak{i}=1}^{\mathfrak{r}}\gamma_{\mathfrak{i}}(x) \bmod~p$, while the $\gp_{l,*}$ correspond to 
        $\displaystyle g(x)=x^{\gcd(n_0,h(p-1))}-(-1)^{hp^k}a_0\equiv \prod_{i=1}^r\gamma_i(x)\bmod~p$. Letting $d_0=\gcd(n_0,h)$ and $d=\gcd\big(n_0,h(p-1)\big)$, the factorization of the ideal generated by $p$ in $\Ocal_{\QQ(\hspace{-.5ex}\sqrt[n]{a})}$ is
\[p \Ocal_{\QQ(\hspace{-.5ex}\sqrt[n]{a})} = \prod_{\mathfrak{i}=1}^{\mathfrak{r}} \gp_{0,\mathfrak{i}}^{\frac{ n_0}{d_0} \cdot p^{m-c} }\prod_{l=1}^c\prod_{i=1}^r\gp_{l,i}^{\frac{ n_0 }{d}\cdot p^{m-c}\varphi(p^l)}.\] \label{OddMainIV}
\end{enumerate}
\end{theorem}



As an application, we classify the common $N$-index divisors of $\QQ\big(\hspace{-.5ex}\sqrt[n]{a}\big)$. Fix $N>0$ and let $K$ be a number field. We say a prime $p$ is a \textit{common $N$-index divisor}, or \textit{C$N$-ID} for brevity, if for any $\big(\theta_1, \dots, \theta_N\big)\in \Ocal_K^N$ the index of $\ZZ\big[\theta_1,\dots, \theta_N\big]$ in $\Ocal_K$ is divisible by $p$ (or infinite in the case where $\QQ\big(\theta_1,\dots, \theta_N\big)\neq K$). In particular, if $p$ is a common $N$-index divisor, then at least $N+1$ ring generators are required to generate $\Ocal_K$ over $\ZZ$. 

For $N=1$, these are the primes whose splitting prevents monogenicity. Indeed the first non-monogenic field was constructed by Dedekind \cite{Dedekind}. He showed that in the cubic field given by $x^3-x^2-2x-8$, the prime 2 is a common index\footnote{Common $1$-index divisors are called \textit{common index divisors}. They are also called \textit{essential discriminant divisors} and \textit{inessential or nonessential discriminant divisors}. The shortcomings of the English nomenclature are partly due to what Neukirch \cite[page 207]{Neukirch} calls ``the untranslatable German catch phrase [...] \textit{au{\ss}erwesentliche Diskriminantenteile}." 
See the final pages of Keith Conrad's exposition \href{https://kconrad.math.uconn.edu/blurbs/gradnumthy/dedekind-index-thm.pdf}{\textit{Dedekind's Index Theorem}} for a detailed explanation of the seemingly contradictory nomenclature. For an excellent historical perspective, see \cite{CIDDBook}.} divisor.  


The work of Pleasants \cite{Pleasants} building on the work of Hensel \cite{Hensel1894} and Dedekind \cite{Dedekind} for $N=1$ shows that common $N$-index divisors can be classified from splitting information. Indeed, Pleasants' main theorem (Theorem \ref{Thm: PleasantsMain}) allows us to use Theorems \ref{Thm: Main} and \ref{Thm: MainOddp} to fully understand common $N$-index divisors in radical extensions. 

First, we let $q=p^N$ and recall that Gauss's formula for the number of monic irreducible polynomials of degree $f$ over $\FF_{q}$ is
\[\Irr(f,q)\coloneqq \frac{1}{f}\sum_{d\mid f}\mu\left(\frac{f}{d}\right)q^d, \text{ where } \mu \text{ is the M\"obius function.}\]


\begin{corollary}[Classification of 2 as a common $N$-index divisor]\label{Cor: 2NCID}
Fix $N>0$ and keep the notation of Theorem \ref{Thm: Main}.
\begin{enumerate}[label=\Roman*., ref=\Roman*]
\item[I. \& II.] If $n$ is odd or if $n$ is even but $2\nmid v_2(a)$, then 2 is not a C$N$-ID of $\QQ\big(\hspace{-.5ex} \sqrt[n]{a}\big)$. \label{CorNCIDsIandII}

\setcounter{enumi}{2}

\item If $n$ is even and $a$ is odd, then let $d_f$ (respectively, $D_f$) be the number of irreducible factors of degree $f$ (respectively, $\frac{f}{2}$) in the factorization of $x^{n_0}-1$ into irreducibles in $\FF_2[x]$ (respectively, $\FF_4[x]$). Note that if $f$ is odd, then $D_f=0$. \label{CorNCIDsIII}
\begin{enumerate}[label=\roman*., ref=III.\roman*]
\item If $w=1$, then 2 is not a C$N$-ID of $\QQ\big(\hspace{-.5ex} \sqrt[n]{a}\big)$. \label{CorNCIDsIIIi}

\item If $2\leq w \leq m+1$, then 2 is a C$N$-ID of $\QQ\big(\hspace{-.5ex} \sqrt[n]{a}\big)$ if and only if 
$\displaystyle D_f + (w-2) \cdot d_f>\Irr\big(f,2^N\big)$ for some $f\in\ZZ^+$. 
\label{CorNCIDsIIIii}

\item If $w > m+1$, then 2 is a C$N$-ID of $\QQ\big(\hspace{-.5ex} \sqrt[n]{a}\big)$ if and only if 
$\displaystyle(m+1) \cdot d_f>\Irr\left(f,2^N\right)$ for some $f\in\ZZ^+$. \label{CorNCIDsIIIiii}
\end{enumerate}

\item If both $a$ and $\gcd\big(v_2(a),n\big)$ are even, then let $d_{f}$ and $D_f$ be as above but corresponding to the factorization of $x^{\gcd(h,n_0)}-1$ into irreducibles.\label{CorNCIDsIV}
\begin{enumerate}[label= Case \O., ref=IV.\O]
    \item If $k\geq m$, then the conclusion is as in \ref{CorNCIDsIII} but with $w_0$ replacing $w$. \label{CorNCIDsIV0}
\end{enumerate}
Suppose now that $k<m$.
\begin{enumerate}[label=\roman*., ref=IV.\roman*]
\item If $w_0=1$, then 2 is a C$N$-ID of $\QQ\big(\hspace{-.5ex} \sqrt[n]{a}\big)$ if and only if $N=1$, $k=1$, and for some 
$f\in\ZZ^+$, $\displaystyle v_2\big(a_0+~1\big)=3$ and $\displaystyle D_f> \Irr(f,2)$, or $\displaystyle v_2\big(a_0+~1\big)\geq 4$  and $\displaystyle 2\cdot d_f > \Irr(f,2).$ \label{CorNCIDsIVi}

\item If $w_0=2$, then 2 is a C$N$-ID of $\QQ\big(\hspace{-.5ex} \sqrt[n]{a}\big)$ if and only if $N=1$ and $\displaystyle D_f>\Irr(f,2)$ for some $f\in\ZZ^+$. \label{CorNCIDsIVii} 

\item If $w_0\geq 3$, then there are four cases: \label{CorNCIDsIViii}

\begin{enumerate}[label=\arabic*., ref=IV.iii.\arabic*]

    \item If $w_0-2<k$, then 2 is a C$N$-ID of $\QQ\big(\hspace{-.5ex} \sqrt[n]{a}\big)$ if and only if\\ $\displaystyle D_f + (w_0-2) \cdot d_f>\Irr\left(f,2^N\right)$ for some $f\in\ZZ^+$. \label{CorNCIDsIViii1new} 

    \item If $w_0-2 \geq k=1$, then 2 is a C$N$-ID of $\QQ\big(\hspace{-.5ex} \sqrt[n]{a}\big)$ if and only if $N=1$ and  $\displaystyle 2\cdot d_f>\Irr\left(f,2\right)$ for some $f\in\ZZ^+$. \label{CorNCIDsIViii2new}

    \item If $w_0-2=k>1$, then 2 is a C$N$-ID of $\QQ\big(\hspace{-.5ex} \sqrt[n]{a}\big)$ if and only if\\ $\displaystyle D_f + (w_0-2) \cdot d_f>\Irr\left(f,2^N\right)$ for some $f\in\ZZ^+$. \label{CorNCIDsIViii3new}
    
    \item If $w_0-2>k>1$, then 2 is a C$N$-ID of $\QQ\big(\hspace{-.5ex} \sqrt[n]{a}\big)$ if and only if\\ $\displaystyle (k+2) \cdot d_f>\Irr\left(f,2^N\right)$ for some $f\in\ZZ^+$. \label{CorNCIDsIViii4} 

\end{enumerate}
\end{enumerate}
\end{enumerate}
\end{corollary}


For odd primes $p$, the following result generalizes the classification of common index divisors in \cite{SmithRadicalSplitting}:  

\begin{corollary}[Classification of $p$ as a common $N$-index divisor]\label{Cor: oddNCIDs}
Let $p$ be an odd integral prime, fix $N>0$, and keep the notation of Theorem~\ref{Thm: MainOddp}.
\begin{enumerate}[label=\Roman*., ref=\Roman*]
\item[I. \& II.] If $p\dnd n$ or $p\mid n$ but $p\dnd v_p(a)$, 
then $p$ is not a C$N$-ID of $\QQ\big(\hspace{-.5ex} \sqrt[n]{a}\big)$. \label{CorOddNCIDsIandII}

\setcounter{enumi}{2}

\item If $p\mid n$ and $p\nmid a$, then let $d_f$ be the number of irreducible factors of degree $f$ in the factorization of $x^{n_0}-a$ into irreducibles in $\FF_p[x]$.
The prime $p$ is a C$N$-ID of $\QQ\big(\hspace{-.5ex} \sqrt[n]{a}\big)$ if and only if 
\[\min(w,m+1) \cdot d_f>\Irr\left(f,p^N\right) \text{ for some } f\in\ZZ^+.\] \label{CorOddNCIDsIII}

\item If $p$ divides $n$, $a$, and $v_p(a)$, then let $\mathfrak{d}_{f}$ be the number of irreducible factors of degree $f$ in the factorization of $x^{\gcd(n_0,h)}-a_0\in \FF_p[x]$ and let $d_f$ be the number of irreducible factors of degree $f$ in the factorization of $x^{\gcd(n_0,h(p-1))}-(-1)^{h}a_0 \in \FF_p[x].$
The prime $p$ is a C$N$-ID of $\QQ\big(\hspace{-.5ex} \sqrt[n]{a}\big)$ if and only if 
\[\mathfrak{d}_{f}+\min(w_0-1,k,m) \cdot d_f > \Irr\left(f,p^N\right) \text{ for some } f\in\ZZ^+.\]\label{CorOddNCIDsIV}
\end{enumerate}
\end{corollary}


\subsection{Organization}

Above we have stated the main theorem (Theorem \ref{Thm: Main}), its complement for odd $p$ (Theorem \ref{Thm: MainOddp} proved in \cite{SmithRadicalSplitting}), and the respective corollaries describing common $N$-index divisors (Corollaries \ref{Cor: 2NCID} and \ref{Cor: oddNCIDs}). 
Despite almost verbatim overlap with \cite{SmithRadicalSplitting}, for the convenience of the reader, we undertake a brief summary of previous work (Subsection~\ref{Sec: PrevWork}) and a description of the initial dissections of the Montes algorithm/Ore's factorization theorem (Section \ref{Sec: Montes}). Section \ref{Sec: ProofMain} is devoted to the proof of Theorem \ref{Thm: Main}. The foundation of this proof is \cite[Theorem 7]{VelezPrimePower}, V\'elez's classification of the splitting of 2 in $\QQ\big(\hspace{-.5ex}\sqrt[2^m]{a}\big)$. Section \ref{Sec: NCIDs} undertakes the proof of Corollaries \ref{Cor: 2NCID} and \ref{Cor: oddNCIDs}. Pleasants's result \cite[Theorem 3]{Pleasants} on the connection between prime splitting and common $N$-index divisors is key here. Novel examples of number rings with common $N$-index divisors are included as well; see Example \ref{Ex: Using2togetNgenerators} and Proposition \ref{Prop: NonRadCommonN}. Section \ref{Sec: Examples} is devoted to a number of examples of Theorem \ref{Thm: Main}. 
Section \ref{Sec: NonMonoNoCIDs} gives a new construction of non-monogenic fields without common index divisors, \textit{exceptional} number fields in Pleasants's terminology. 
The construction 
is distinctive in that it uses prescribed vanishing of terms in the index form to create exceptional radical extensions of arbitrary degree $n>2$.


\subsection{Previous work}\label{Sec: PrevWork}

There has been a good deal of recent work on common 1-index divisors. See \cite{MR2251710}, \cite{MR3889325}, \cite{MR4736305}, \cite{MR4563435}, \cite{MR2531162}, \cite{MR4510645}, \cite{MR731633}, as well as other works from these authors for various families of polynomials/fields of a fixed degree. When consulting the literature, it bears noting that the question of whether or not a given a binomial (radical) polynomial or a trinomial polynomial is monogenic has been fully answered. See \cite{g17} and \cite{JKS1}, respectively. For families of varying degree, \cite{PethoPohst} study multiquadratic fields, while \cite{SpearmanWilliamsYang} study fields generated by polynomials of prime degree with dihedral Galois groups. 

\cite{CIDDBook}, the recent book by Gouv\^ea and Webster, provides a vibrant historical account of the impact of common index divisors on the development of algebraic number theory. Further, the book includes new translations of two of the seminal papers already referenced: \cite{Dedekind} and \cite{Hensel1894}.

We note that \cite{Berwick1927} uses Newton polygon techniques to establish an integral basis for $\QQ(\hspace{-.5ex} \sqrt[n]{a})$. In some cases, this work can describe the splitting of primes dividing $na$, but it does not fully describe splitting.  
The general method is subsumed by earlier work of Ore \cite{OreReview}. 

In \cite{ObusWild}, the author computes bounds on the conductors of extensions obtained by a root of unity and a radical. In a particular case when $p=2$, exact values of the conductor are computed. \cite{Viviani} constructs a uniformizer for $\QQ_p\big(\zeta_p, \sqrt[p^{i+1}]{a}\big)/\QQ_p\big(\zeta_p, \sqrt[p^{i}]{a}\big)$. \cite{BellemareLei} describes a uniformizer for the extension $\QQ_p\big(\zeta_{p^2},\sqrt[p]{p}\big)/\QQ_p$, while \cite{WangYuan} continues down this avenue by constructing a uniformizer for $\QQ_p\big(\zeta_{p^2},\sqrt[p^m]{p}\big)/\QQ_p$ with $m\geq 1$. 

In \cite{Velez1/p}, V\'elez describes the factorization of a prime above $p$ in any extension of a number field obtained by adjoining a $p^{\text{th}}$ power or a $p^{\text{th}}$ root of unity. In the series of papers \cite{MannVelez} and \cite{VelezCoprime}, prime ideal splitting in radical extensions is described in a variety of cases. 
In \cite{VelezPrimePower}, V\'elez completely describes the splitting of the prime $p$ in a $p$-power radical extension of $\QQ$. 
This work will be our jumping off point (Theorem \ref{Thm: Velez2}) in addressing the more difficult cases (Cases \ref{MainIII} and \ref{MainIV}) of Theorem \ref{Thm: Main}.
V\'elez's primary application was describing the genus field of $\QQ\big(\hspace{-.5ex} \sqrt[n]{a}\big)$, while our application is the classification of common $N$-index divisors of $\QQ\big(\hspace{-.5ex} \sqrt[n]{a}\big)$. Thus, we have rephrased and slightly reinterpreted V\'elez's results in order to suit our goals. 


\section{The Montes algorithm and a theorem of Ore}\label{Sec: Montes}


The Montes algorithm is a $p$-adic factorization algorithm that is based on and extends the pioneering work of {\O}ystein Ore \cite{Ore}. We will essentially only employ the aspects developed by Ore here, but we will use the notation and setup of the general implementation. For the complete development of the Montes algorithm, see \cite{GMN}. Our notation will roughly follow \cite{ElFadilMontesNart}, which gives a more extensive summary than we undertake here. One can also consult \cite{JhorarKhanduja}.

Let $p$ be an integral prime, $K$ a number field with ring of integers $\Ocal_K$, and $\gp$ a prime of $K$ above $p$. Write $K_\gp$ to denote the completion of $K$ at $\gp$. By a \textit{uniformizer at $\gp$} or a \textit{uniformizer of $K_\gp$}, we mean an element $\pi_\gp\in\Ocal_K$ such that $v_\gp\big(\pi_\gp\big)=1$. 
Suppose we have a monic, irreducible polynomial $f(x)\in \Ocal_K[x]$. We extend the standard $\gp$-adic valuation to $\Ocal_K[x]$ by defining the $\gp$-adic valuation of $f(x) = a_n x^n + \cdots + a_1 x + a_0 \in \Ocal_K[x]$ to be 
	\[ v_\gp\big(f(x)\big) = \min_{0 \leq i \leq n} \big( v_\gp(a_i) \big). \]
This is sometimes called the \textit{Gauss valuation.}
If $\phi(x), f(x) \in \Ocal_K[x]$ are monic and such that $\deg \phi \leq \deg f$, then we can write
    	\[f(x)=\sum_{i=0}^k a_i(x)\phi(x)^i,\]
for some $k$, where each $a_i(x) \in \Ocal_K[x]$ has degree less than $\deg \phi$. We call the above expression the \emph{$\phi$-adic development} of $f(x)$. We associate to the $\phi$-adic development of $f(x)$ an open Newton polygon by taking the lower convex hull
of the integer lattice points $\big(i,v_p(a_i(x))\big)$. The sides of the Newton polygon with negative slope are the \emph{principal $\phi$-polygon}. 

Write $k_\gp$ for the residue field $\Ocal_K/\gp$, and let $\ol{f(x)}$ be the image of $f(x)$ in $k_\gp[x]$. It will often be the case that we develop $f(x)$ with respect to an irreducible factor $\phi(x)$ of $\ol{f(x)}$. In this situation, we will want to consider the extension of $k_\gp$ obtained by adjoining a root of $\phi(x)$. We denote this finite field by $k_{\gp,\phi}$. We associate to each side of the principal $\phi$-polygon a polynomial in $k_{\gp,\phi}[y]$. Suppose $S$ is a side of the principal $\phi$-polygon with initial vertex $\big(s,v_\gp(a_s(x))\big)$, terminal vertex $\big(k,v_\gp(a_k(x))\big)$, and slope $-\frac{h}{e}$ written in lowest terms. Define the length of the side to be $l(S)=k-s$ and the degree to be $d\coloneqq\frac{l(S)}{e}$. Let $\red:\Ocal_K[x]\to k_{\gp,\phi}$ denote the homomorphism obtained by quotienting by the ideal $\big(\gp,\phi(x)\big)$.
For each $i$ in the range $b\leq i\leq k$, we define the residual coefficient to be
\[c_i=\left\{
\begin{array}{ll}
0 \text{ if }  \big(i,v_\gp(a_i(x))\big)  \text{ lies strictly above } S  \text{ or } v_\gp(a_i(x))=\infty,\\
\red\left(\frac{a_i(x)}{\pi^{v_\gp(a_i(x))}}\right)  \text{ if }  \big(i,v_\gp(a_i(x))\big) \text{ lies on } S.
\end{array}
\right.\]
Finally, the \emph{residual polynomial} of the side $S$ is the polynomial
\[R_S(y)=c_s+c_{s+e}y+\cdots +c_{s+(d-1)e}y^{d-1}+c_{s+de}y^d\in k_{\gp,\phi}[y].\]
Notice that $c_s$ and $c_{s+de}$ are always nonzero since they are the initial and terminal vertices, respectively, of the side $S$. In this work, we will almost always be developing $f(x)$ with respect to a linear polynomial, so $k_{\gp,\phi}=k_\gp$, and we will often write the latter to ease notation.


Having established notation, we state a theorem that connects prime splitting and polynomial factorization. The ``three dissections" that we will outline below are due to Ore, and the full Montes algorithm is an extension of this. Our statement loosely follows Theorem 1.7 of \cite{ElFadilMontesNart}. 

\begin{theorem}\label{Thm: Ore}[Ore's Three Dissections]
Let $f(x)\in \Ocal_K[x]$ be a monic irreducible polynomial and let $\alpha$ be a root. Suppose
\[\ol{f(x)}=\phi_1(x)^{r_1}\cdots \phi_s(x)^{r_s}.\]
is a factorization into irreducibles in $k_\gp[x]$. Hensel's lemma shows $\phi_i(x)^{r_i}$ corresponds to a factor of $f(x)$ in $K_\gp[x]$ and hence to a factor $\gm_i$ of $\gp \Ocal_{K(\alpha)}$. 

Choose a lift of $\phi_i(x)$ to $\Ocal_K[x]$ and, abusing notation, call this lift $\phi_i(x)$. Developing $f(x)$ with respect to $\phi_i(x)$, suppose the principal $\phi_i$-polygon has sides $S_1,\dots, S_g$. Each side of this polygon corresponds to a distinct factor of $\gm_i$. 

Write $\gn_j$ for the factor of $\gm_i$ corresponding to the side $S_j$. Suppose $S_j$ has slope $-\frac{h}{e}$. If the residual polynomial $R_{S_j}(y)$ is separable, then the prime factorization of $\gn_j$ mirrors the factorization of $R_{S_j}(y)$ in $k_{\gp,\phi_i}[y]$, but every factor of $R_{S_j}(y)$ will have an exponent of $e$. In other words,
\[\text{if } R_{S_j}(y)=\gamma_1(y)\dots\gamma_k(y) \ \text{ in }  \ k_{\gp,\phi_i}[y], \ \text{ then } \ \gn_j = \gP_1^{e}\cdots \gP_k^{e}  \ \text{ in } \ \Ocal_{K(\alpha)},\]
with $\deg(\gamma_m)$ equaling the residue class degree of $\gP_m$ over $k_{\gp,\phi_i}$ for each $1\leq m\leq k$.
In the case where $R_{S_j}(y)$ is not separable, further developments are required to factor $\gp$. 
\end{theorem}

Though Theorem \ref{Thm: Ore} describes general $\gp$-adic factorizations, we will be working with monic integral polynomials, so our factorizations will be over the $\gp$-adic integers. 


In addition to Ore's Theorem, it will be convenient to have the following classical result on the irreducibility of radical polynomials. For a reference see \cite{Turnwald} or Chapter 6, \S 9 of \cite{Algebra}.

\begin{theorem}[The Vahlen--Capelli Theorem]\label{Thm: VahlenCapelli}
Let $K$ be a field and $x^n-a\in K[x]$. Then $x^n-a$ is reducible over $K$ if and only if for some prime $p\mid n$ we have $a\in K^p$ or $4\mid n$ and $a\in -4K^4$.
\end{theorem}


\section{Proof of the Main Theorem}\label{Sec: ProofMain}


This section is devoted to the proof of Theorem \ref{Thm: Main}. We will proceed through the cases in numerical order. In contrast to the statement of Theorem \ref{Thm: Main}, in this section the choice of $i$ or $j$ for an index does not indicate whether a factorization is occurring over $\FF_2$ or $\FF_4$. Likewise, we will freely conflate a polynomial in $\ZZ[x]$ with its image in $\FF_p[x]$, $\QQ_p[x]$, and extensions thereof. Throughout we readily use the correspondence between $p$-adic factors of an irreducible polynomial over a number field and prime ideal factors of $p$ in the associated extension. See Proposition II.8.2 of \cite{Neukirch}. Finally, to avoid being overly verbose, we often speak of \textit{the factorization of $\gp$ in $K$} when we mean \textit{the prime ideal factorization of the ideal generated by $\gp$ in the ring of integers of $K$}. 

First, Case \ref{MainI} is simply Dedekind--Kummer factorization. See \cite[I.\S 8. Proposition 25]{LangANT} or Proposition I.8.3 of \cite{Neukirch}.

Case \ref{MainII} is covered by Theorem 6.1 of \cite{SmithRadicalSplitting}. The proof is an application of Theorem \ref{Thm: Ore}. 

For Cases \ref{MainIII} and \ref{MainIV} we use \cite[Theorem 7]{VelezPrimePower}. 
We state the result here for convenience and to unify notation. Note that we have fixed a minor typo in V\'elez's Case D) ii) c) (Case \ref{VelezEvenIIIi} for us), and we have included the factorization of $x^{2^m}-a$ over $\ZZ_2[x]$ for reference. Further, we have the theorem for an arbitrary even $a$ as opposed to an $a$ such that $v_2(a)=2^k$. This is justified in Remark \ref{Rmk: ArbEvena}.

\begin{theorem}[Theorem 7 of \cite{VelezPrimePower}]\label{Thm: Velez2}
        Suppose $x^{2^m}-a\in \ZZ[x]$ is irreducible, and let $K=\QQ\big(\hspace{-.5ex}\sqrt[2^m]{a}\big)$. The following cases describe the prime ideal factorization of $2\Ocal_K$. If $a$ is odd, then write $w=v_2\big(a^{2^m}-a\big)=v_2(a-1)$. 
    \begin{enumerate}[label=\roman*., ref=\roman*]
        \item If $w=1$, then 2 is totally ramified: $\displaystyle 2\Ocal_K =\gp^{2^m}.$ In this case, $x^{2^m}-a$ is irreducible in $\ZZ_2[x]$. \label{VelezOddi}
        
        \item If $2\leq w\leq m+1$, then $\displaystyle 2\Ocal_K =\prod_{i=1}^{w-1}\gp_i^{2^{m-w+i}},$
        where $\gp_1$ has residue class degree 2 and $\gp_{i>1}$ has residue class degree 1. In this case, $a=\alpha^{2^{w-2}}$ for some $\alpha\in\ZZ_2$ that is not a square. The factorization of $x^{2^m}-a$ in $\ZZ_2[x]$ is 
        \[x^{2^m}-a=\left(x^{2^{m-w+2}}-\alpha\right)\left(x^{2^{m-w+2}}+\alpha\right)\left(x^{2^{m-w+3}}+\alpha^2\right)\cdots\left(x^{2^{m-1}} + \alpha^{w-3}\right),\]
        deleting terms from the right as necessary.\label{VelezOddii}
        
        \item If $w>m+1$, then $\displaystyle 2\Ocal_K =\gp_0\prod_{i=1}^{m}\gp_i^{2^{i-1}}.$ 
        Each prime has residue class degree 1. In this case, $a\in \ZZ_2^{2^m}$, so the factorization of $x^{2^m}-a$ agrees with the factorization of $x^{2^m}-1=(x-1)(x+1)(x^2+1)\cdots\big(x^{2^{m-1}}+1\big)$.\label{VelezOddiii}
    \end{enumerate}
        If $a$ is even, write $a=a_02^{h2^k}$ as in Theorem \ref{Thm: Main}. 
        When $k=0$, $x^{2^m}-a$ is 2-Eisenstein. Thus, the polynomial is irreducible and we have total ramification at 2. Assume $k>0$ and write $w_0=v_2(a_0-1)$ and $t=\max(0,w_0-2)$. By Lemma \ref{Lem: 2mthroot}, $a_0\in \ZZ_2^{2^t}$ but $a_0\notin \ZZ_2^{2^{t+1}}$. 
        \begin{enumerate}[label=\Roman*., ref=\Roman*]
            \item If $w_0=1$, then there are a number of subcases. Note $v_2(a_0+1)>1$ since $w_0=1$. \label{VelezEvenI}
                \begin{enumerate}[label=\arabic*., ref=I.\arabic*]
                    \item If $v_2(a_0+1)=2$, then $2\Ocal_K=\gp^{2^{m}}$ and $x^{2^m}-a$ is irreducible in $\ZZ_2[x]$. \label{VelezEvenIi}
                    \item If $v_2(a_0+1)\geq 3$ and $k\geq 2$, then $2\Ocal_K=\gp^{2^{m}}$ and $x^{2^m}-a$ is irreducible in $\ZZ_2[x]$.\label{VelezEvenIii}
                    \item If $v_2(a_0+1)=3$ and $k=1$, then $2\Ocal_K=\gp^{2^{m-1}}$ with $f(\gp/2)=2$ and $x^{2^m}-a$ is irreducible in $\ZZ_2[x]$. \label{VelezEvenIiii} 
                    \item If $v_2(a_0+1)\geq 4$ and $k=1$, then $2\Ocal_K=\gp_1^{2^{m-1}}\gp_2^{2^{m-1}}$. Since $v_2(a_0+1)\geq 4$, there exists $\beta\in \ZZ_2$ such that $\beta^4=-a_0$, and 
                    \[x^{2^m}-a=\big(x^{2^{m-1}}+2\beta x^{2^{m-2}}+2\beta^2\big)\big(x^{2^{m-1}}-2\beta x^{2^{m-2}}+2\beta^2\big)\] is the corresponding factorization into irreducibles in $\ZZ_2[x]$.\label{VelezEvenIiv}
                \end{enumerate}
            \item If $w_0=2$, then $2\Ocal_K=\gp^{2^{m-1}}$ with $f(\gp/2)=2$ and $x^{2^m}-a$ is irreducible in $\ZZ_2[x]$. \label{VelezEvenII} 
            \item If $w_0\geq 3$, then $\alpha\coloneqq\sqrt[2^t]{a_0}\in \ZZ_2$ with $t\geq 1$. In this case, we have the following: \label{VelezEvenIII}
                \begin{enumerate}[label=\arabic*., ref=III.\arabic*]
                    \item If $t < k$, then $\displaystyle 2\Ocal_K=\left(\gp_0\gp_1^{2}\gp_2^{4}\cdots \gp_t^{2^{t}}\right)^{2^{m-t-1}} \text{ with } f(\gp_0/2)=2 , \  f(\gp_{i>0}/2)=1.$
                    The corresponding factorization into irreducibles in $\ZZ_2[x]$ is
                    \[x^{2^m}-a=\left(x^{2^{m-t}}-2^{2^{k-t}}\alpha\right)\prod_{i=0}^{t-1}\left(x^{2^{m-t+i}}+2^{2^{k-t+i}}\alpha^{2^{i}}\right).\] \label{VelezEvenIIIi}
                    \item If $t=k$, then  $\displaystyle 2\Ocal_K=\left(\gp_0\gp_1\gp_2\gp_3^4\cdots \gp_t^{2^{t-1}}\right)^{2^{m-t}} \text{ with } f(\gp_2/2)=2 , \ f(\gp_{i\neq 2}/2)=1.$
                    The corresponding factorization into irreducibles in $\ZZ_2[x]$ is
                    \[x^{2^m}-a=\left(x^{2^{m-t}}-2\alpha\right)\prod_{i=0}^{t-1}\left(x^{2^{m-t+i}}+2^{2^{i}}\alpha^{2^{i}}\right).\] \label{VelezEvenIIIii}
                    \item If $t\geq k+1$ and $m=2$, then $k=1$ and $2\Ocal_K=\gp_0^2\gp_1^2$. Further, $x^4-a=\big(x^2-2\alpha^{2^{t-1}}\big)\big(x^2+2\alpha^{2^{t-1}}\big)$ is the corresponding factorization into irreducibles. \label{VelezEvenIIIiii}
                    \item If $t\geq k+1$ and $m\geq 3$, then $2\Ocal_K=\big(\gp_0\gp_1\gp_2\gp_2'\gp_3^4\cdots \gp_k^{2^{k-1}}\big)^{2^{m-k}}$ where each prime has residue class degree 1. The corresponding factorization into irreducibles in $\ZZ_2[x]$ is \label{VelezEvenIIIiv}
                    \begin{align*}
                    x^{2^m}-a=\left(x^{2^{m-k}}-2\alpha^{2^{t-k}}\right)&\left(x^{2^{m-k}}+2\alpha^{2^{t-k}}\right)\left(x^{2^{m-k}}-2\alpha^{2^{t-k-1}}x^{2^{m-k-1}}+2\alpha^{2^{t-k}}\right)\\
                    &\cdot\left(x^{2^{m-k}}+2\alpha^{2^{t-k-1}}x^{2^{m-k-1}}+2\alpha^{2^{t-k}}\right)\prod_{i=2}^{k-1}\left(x^{2^{m-k+i}}+2^{2^{i}}\alpha^{2^{t-k+i}}\right).
                    \end{align*}
                \end{enumerate}
        \end{enumerate}
        \end{theorem}


Returning to the proof of Theorem \ref{Thm: Main}, Theorem \ref{Thm: Velez2} gives the factorization of $2$ in the ring of integers of $\QQ\big(\hspace{-.5ex}\sqrt[2^m]{a}\big)$. Recall $n=2^m n_0$ with $n_0$ odd. Our strategy is to describe the factorization of each prime of $\Ocal_{\QQ(\hspace{-.5ex}\sqrt[2^m]{a})}$ above 2 in the ring of integers of $\QQ\big(\hspace{-.5ex}\sqrt[2^m]{a},\sqrt[n_0]{a}\big)=\QQ\big(\hspace{-.5ex}\sqrt[n]{a}\big)$; this is accomplished by Lemma \ref{Lem: Split2nmida} for Case \ref{MainIII} and Lemma \ref{Lem: Split2an0} for Case \ref{MainIV}. The following diagram summarizes this:

\begin{figure}[h!]
\xymatrix{
&&&\QQ\big(\hspace{-.5ex}\sqrt[2^m]{a},\sqrt[n_0]{a}\big)=\QQ\big(\hspace{-.5ex}\sqrt[n]{a}\big)\ar@{-}[d]^{\text{Splitting given by Lemmas \ref{Lem: Split2nmida} \& \ref{Lem: Split2an0}}} & & \gp\Ocal_{\QQ(\hspace{-.5ex}\sqrt[n]{a})} =\prod\limits_{j=1}^s\gP_j^{h_j}\ar@{-}[d]\\
&&&\QQ\big(\hspace{-.5ex}\sqrt[2^m]{a}\big)\ar@{-}[d]^{\text{Splitting given by Theorem \ref{Thm: Velez2}}} & & 2\Ocal_{\QQ(\hspace{-.5ex}\sqrt[2^m]{a})}=\prod\limits_{i=1}^{r}\gp_i^{e_i}\ar@{-}[d]\\
&&&\QQ && 2\ZZ\\
}
\caption{Roadmap for the proof of Cases \ref{MainIII} and \ref{MainIV}}
\label{fig. split2}
\end{figure}

In the end, Figure \ref{fig. split2} shows 
\[2\Ocal_{\QQ(\hspace{-.5ex}\sqrt[n]{a})} = \prod\limits_{i=1}^{r}\left(\prod\limits_{j=1}^{s_i}\gP_{i,j}^{h_{i,j}}\right)^{e_i}.\]

\begin{lemma}\label{Lem: Split2nmida}
    Suppose $2\mid n$ and $2\nmid a$. 
    Let $\gp\mid 2\Ocal_{\QQ(\hspace{-.5ex}\sqrt[2^m]{a})}$ be a prime of $\Ocal_{\QQ(\hspace{-.5ex}\sqrt[2^m]{a})}$, and write $k_\gp$ 
    for the residue field of $\gp$. If 
    \begin{equation}\label{Eq: FactorsDK}
           x^{n_0}-1=\prod_{i=1}^r\gamma_i(x) 
    \end{equation}
    is a factorization into irreducibles in $k_\gp[x]$, then the factorization of $\gp$ in $\Ocal_{\QQ(\hspace{-.5ex}\sqrt[n]{a})}$ is given by
    \[\gp \Ocal_{\QQ(\hspace{-.5ex}\sqrt[n]{a})}=\prod_{i=1}^r \gP_i, \text{ where the degree of the residue field of $\gP_i$ over $k_\gp$ is } \deg \gamma_i(x).\]
\end{lemma}


\begin{proof}
    Since $\gp\mid 2\Ocal_{\QQ(\hspace{-.5ex}\sqrt[2^m]{a})}$ and $2\nmid \Disc\big(x^{n_0}-a\big)$, we can apply the Dedekind--Kummer theorem (see \cite[I.\S 8. Proposition 25]{LangANT}) to see that the factorization of $\gp$ in the ring of integers of $\QQ\big(\hspace{-.5ex}\sqrt[2^m]{a},\sqrt[n_0]{a}\big)=\QQ\big(\hspace{-.5ex}\sqrt[n]{a}\big)$ mirrors the factorization of $x^{n_0}-a$ in $k_\gp[x]$. 
\end{proof}


Applying Lemma \ref{Lem: Split2nmida} to Cases \ref{VelezOddi}, \ref{VelezOddii}, and \ref{VelezOddiii} of Theorem \ref{Thm: Velez2}, we have Case \ref{MainIII} of Theorem \ref{Thm: Main}.
For Case \ref{MainIV}, Dedekind--Kummer factorization is not enough, and we must employ Theorem \ref{Thm: Ore}. 
Note Theorem \ref{Thm: Velez2} Cases \ref{VelezEvenI}, \ref{VelezEvenII}, and \ref{VelezEvenIII} (or Cases \ref{VelezOddi}, \ref{VelezOddii}, and \ref{VelezOddiii} if $k=v_2\big(v_2(a)\big)\geq v_2(n)=m$) describe the splitting of 2 in the ring of integers of $\QQ\big(\hspace{-.5ex}\sqrt[2^m]{a}\big)$.


\begin{lemma}\label{Lem: Split2an0}
    Suppose $a$ is even, and let $\gp\mid 2\Ocal_{\QQ(\hspace{-.5ex}\sqrt[2^m]{a})}$ be a prime of $\Ocal_{\QQ(\hspace{-.5ex}\sqrt[2^m]{a})}$. Write $k_\gp$ 
    for the residue field of $\gp$ and let $d=\gcd(v_2(a),n_0)$. If 
    \begin{equation}\label{Eq: FactorsR(y)}
        x^{d}-1=\prod_{j=1}^s\gamma_j(x) 
    \end{equation}
    is a factorization 
    into irreducibles in $k_\gp[x]$, then the factorization of $\gp$ in $\Ocal_{\QQ(\hspace{-.5ex}\sqrt[n]{a})}$ is given by
    \[\gp \Ocal_{\QQ(\hspace{-.5ex}\sqrt[n]{a})}=\prod_{j=1}^s\gP_j^{\frac{n_0}{d}}, \text{ where the degree of the residue field of $\gP_j$ over $k_\gp$ is } \deg \gamma_j(x).\]
\end{lemma}


\begin{proof}
Let $e$ be the ramification index of $\gp$ over $2$, so $v_\gp(2)=e$. Note that Theorem \ref{Thm: Velez2} shows that $e$ is a power of 2. Letting $K_\gp$ denote the extension of $\QQ_2$ corresponding to $\gp$, we begin to apply the Montes algorithm to factor $x^{n_0}-a$ in $K_\gp[x]$ and equivalently $\gp$ in $\QQ\big(\hspace{-.5ex}\sqrt[n]{a}\big)$.

The reduction of $x^{n_0}-a$ modulo $\gp$ yields $x^{n_0}$. The $x$-adic development is $x^{n_0}-a$, and the single side of the principal $x$-polygon has slope $-\frac{ev_2(a)}{n_0}$. Since $e$ is a power of 2 and $n_0$ is odd, we see that $\gp$ has ramification index $\frac{n_0}{d}$ in $\QQ\big(\hspace{-.5ex}\sqrt[n]{a}\big)$. 
Letting $\pi_\gp$ be a uniformizer for $K_\gp$, the residual polynomial is $R(y)=y^{d}-a/\pi_\gp^{ev_2(a)}$, and Theorem \ref{Thm: Velez2} shows that the residue field of $\gp$ is either $\FF_2$ or $\FF_4$. 

In the case where $k_\gp=\FF_2$, we see $R(y)=y^{d}-1$ in $\FF_2[y]$. Since the degree is odd, $R(y)$ is separable. Thus the factorization of $\gp$ mirrors the factorization of $R(y)$, but the ramification index of each factor is multiplied by $\frac{n_0}{d}$.

In the case where $k_\gp=\FF_4$, it is not obvious that the factorization of $R(y)=y^{d}-a/\pi_\gp^{ev_2(a)}$ agrees with $x^d-1$. We approach the problem in another way. First, we consider $y^d-a$. Changing variables, we let $y\mapsto 2^{v_2(a)/d}y$, so our polynomial becomes $y^d-a/ 2^{v_2(a)}=y^d-a_0$. As the discriminant is relatively prime to 2 and $a_0\in \ZZ$, we reduce to obtain $y^d-a_0\equiv y^d-1\bmod \gp$, and Dedekind--Kummer factorization shows the splitting of $\gp$ in $\QQ\big(\hspace{-.5ex}\sqrt[2^m]{a},\sqrt[d]{a}\big)$ mirrors the factorization of $y^d-1\equiv \prod_{j=1}^s\gamma_j(x) \bmod \gp$. Note that $\gp$ is unramified in this extension. From our earlier work, we see that the ramification index of $\frac{n_0}{d}$ must be obtained in the extension to $\QQ\big(\hspace{-.5ex}\sqrt[n]{a}\big)$.

More explicitly, we can adjoin a root of $x^\frac{n_0}{d}-\sqrt[d]{a}$ to $\QQ\big(\hspace{-.5ex}\sqrt[2^m]{a},\sqrt[d]{a}\big)$ to obtain $\QQ\big(\hspace{-.5ex}\sqrt[n]{a}\big)$. The principal $x$-polygon is one-sided with slope $-\frac{ev_2(a)/d}{n_0/d}$. Thus the extension is totally ramified of degree $\frac{n_0}{d}$.
\end{proof}

We apply Lemma \ref{Lem: Split2an0} to each of the subcases of Cases \ref{VelezEvenI}, \ref{VelezEvenII}, and \ref{VelezEvenIII} of Theorem \ref{Thm: Velez2}, and we obtain Cases \ref{MainIVi}, \ref{MainIVii}, and \ref{MainIViii} of Theorem \ref{Thm: Main}. If $k=v_2\big(v_2(a)\big)\geq v_2(n)=m$, then $\QQ\big(\hspace{-.5ex}\sqrt[2^m]{a}\big)=\QQ\big((a_02^{2^kh})^\frac{1}{2^m}\big)=\QQ\big(\hspace{-.5ex}\sqrt[2^m]{a_0}\big)$, so we employ Cases \ref{VelezOddi}, \ref{VelezOddii}, and \ref{VelezOddiii} with $x^{2^m}-a_0$ to obtain the splitting of 2 in $\QQ\big(\hspace{-.5ex}\sqrt[2^m]{a}\big)$. We then apply Lemma \ref{Lem: Split2an0} to obtain Theorem \ref{Thm: Main} Case \ref{MainIV0}. This completes our proof of Theorem \ref{Thm: Main}.



\subsection{Auxiliary Lemmas and Remarks}


For Theorem \ref{Thm: Velez2}, it is important to know the largest non-negative integer $t$ such that $a_0$, the odd part of $a$, is in $\ZZ_2^{2^t}$ (equivalently, $\QQ_2^{2^t}$). Asking whether $a_0$ is a $2^t$-th power in $\ZZ_2$ is the same as asking whether $x^{2^t}-a_0$ has a root in $\ZZ_2$. We employ Theorem \ref{Thm: Ore}. Reducing $x^{2^m}-a_0$ modulo 2, we obtain $x^{2^m}-a_0\equiv \big(x-a_0\big)^{2^m}\equiv (x-1)^{2^m}\bmod 2$. Thus we take the $(x-1)$-adic development: 
\begin{equation}\label{Eq: binomexpansion2}
\begin{split}
x^{2^m}-a_0&=\left(x-1+1\right)^{2^m}-a_0\\
&=\left(\sum\limits_{i=0}^{2^m}\binom{2^m}{i}\left(x-1\right)^i 1^{2^m-i}\right) - a_0\\
&=\left(\sum\limits_{i=1}^{2^m}\binom{2^m}{i} \left(x-1\right)^i\right) + 1 - a_0.
\end{split}
\end{equation}


Let $w_0=v_2(a_0-1)$ and let $S$ denote the left-most side of the principal $(x-1)$-polygon. The side $S$ identifies whether we have a root in $\ZZ_2$. 
The initial vertex is $(0,w_0)$. If $w_0\leq m$, then $S$ has a fractional slope and does not correspond to a root. If $w_0>m+1$, then $S$ has length 1 and we have a root. If $w_0=m+1$, then $S$ has length 2 and contains the integer lattice point $(1,m)$. More analysis is needed in this case. These three cases are illustrated in Figure \ref{Fig: tandtheleftmost}. 


\begin{figure}[h!]
\begin{tikzpicture}
      \draw[<->] (0,5.5) -- (0,0) -- (4.5,0);
      \draw [fill] (0,5) circle [radius = .05];
      \draw [fill] (0,4) circle [radius = .05];
      \draw [fill] (0,3) circle [radius = .05];
      \draw [fill] (1,3) circle [radius = .05];  
      \draw [fill] (2,2) circle [radius = .05];  
      \node [right] at (.95,3.2) {$(1,m)$};
      \node [right] at (1.95,2.2) {$(2,m-1)$};
      \node [left] at (-.05,5.2) {$w_0=m+2$};
      \node [left] at (-.05,4.2) {$w_0=m+1$};
      \node [left] at (-.05,3.2) {$w_0=m$};

      \foreach \x in  {1,2,3,4} {
      		\draw (\x, 2pt) -- +(0,-4pt);
      }
      \foreach \y in  {1,2,3,4,5} {
      		\draw (2pt, \y) -- +(-4pt, 0);
      }
      
      \draw (1,3) -- (2,2);
      \draw (0,4) -- (1,3);
      \draw (0,5) -- (1,3);
      \draw (0,3) -- (2,2);
	\end{tikzpicture}
  \caption{Possibilities for the left-most side as $w_0$ varies}
  \label{Fig: tandtheleftmost}
\end{figure}


If $w_0=m+1$, then $(2,m-1)$ is the terminal vertex of $S$. We analyze the residual polynomial in $\FF_2[y]$,
\begin{align*}
    R_S(y)&=\frac{1-a_0}{2^{w_0}}+\frac{2^m}{2^m}y+\frac{\binom{2^m}{2}}{2^{m-1}}y^2\\
    &=y^2+y+1.
\end{align*}

We see that the factor of $x^{2^m}-a_0$ corresponding to the side $S$ yields a degree 2 extension of $\FF_2$ and not a root in $\ZZ_2$. Thus $x^{2^m}-a_0$ does not have any roots in $\ZZ_2$ and $a_0\notin \ZZ_2^{2^m}$. (Equivalently, $a_0\notin \QQ_2^{2^m}$.) We summarize with the following lemma.

\begin{lemma}\label{Lem: 2mthroot}
Let $a_0\in \ZZ$ be odd and write $w_0=v_2\big(a_0^{2^m}-a_0\big)=v_2\big(a_0^{2}-a_0\big)=v_2\big(a_0-1\big)$. If $w_0=1,2$, then $a_0\notin \ZZ_2^2$. Otherwise, $w_0>2$ and $a_0\in \ZZ_2^{2^{w_0-2}}$ but $a_0\notin \ZZ_2^{2^{w_0-1}}$. In other words, if $t=\max(0,w_0-2)$, then $a_0\in\ZZ_2^{2^t}$ but $a_0\not\in\ZZ_2^{2^{t+1}}.$
\end{lemma}


\begin{remark}\label{Rmk: ArbEvena}
V\'elez states Theorem \ref{Thm: Velez2} for $a=a_12^{2^k}$ with $a_1$ odd; however, our statement is equivalent. Indeed, suppose $a$ is even and $2\mid\gcd\big(v_2(a),n\big)$, so we can write $a$ as $a=a_02^{h2^k}$ with $h$ and $a_0$ odd. Using B\'ezout's identity, we can find $u,v\in\ZZ$ such that $hu+2^mv=1$. Here $u$ is odd, so $\QQ\big(\hspace{-.5ex}\sqrt[2^m]{a}\big)=\QQ\big(\hspace{-.5ex}\sqrt[2^m]{a^u}\big)$. We find that 
\[a^\frac{u}{2^m}=a_0^\frac{u}{2^m}2^\frac{hu2^k}{2^m}=a_0^\frac{u}{2^m}2^\frac{(1-2^mv)2^k}{2^m}=\frac{a_0^\frac{u}{2^m}2^\frac{2^{k}}{2^m}}{2^v}.\]
Thus, letting $a_1=a_0^u$, it suffices to consider $\QQ\big(\hspace{-.5ex}\sqrt[2^m]{a_12^{2^k}}\big)$. 

We are interested in the largest $t$ such that $a_12^{2^k}\in \ZZ_2^{2^t}$ and $a_12^{2^k}\notin \ZZ_2^{2^{t+1}}$. We claim this is the same as the largest $t$ such that $a\in \ZZ_2^{2^t}$ and $a\notin \ZZ_2^{2^{t+1}}$. 

To see this, it suffices to show that the relevant $t$-values for $a_0$ and $a_1$ agree. Notice that if $u$ is odd and $a_0=A^{2^t}$ for some $A$ that is not a square, then $a_0^u=\big(A^{2^t}\big)^u=\big(A^u\big)^{2^t}$. Thus $a_0^u\in \ZZ_2^{2^t}$. Conversely, if $a_0^u=B^{2^t}$, then $a_0=B^{2^t}/a_0^{u-1}$. Since $u-1$ is even, $a_0$ is a square in $\QQ_2$ and hence in $\ZZ_2$. We have $\sqrt{a_0}^{u}=B^{2^{t-1}}$. If $t>1$, then we have $\sqrt{a_0}=B^{2^{t-1}}/\sqrt{a_0}^{u-1}$, so that $\sqrt{a_0}$ is a square in $\QQ_2$ and hence $\ZZ_2$. Continuing in this manner, we see that $a_0$ is a $2^t$-th power. Thus $a_0\in \ZZ_2^{2^t}\iff a_1\in \ZZ_2^{2^t}$, and 
\[w_0=v_2(a_0-1)=v_2(a_1-1).\]
\end{remark}


\section{Common $N$-Index Divisors}\label{Sec: NCIDs}

With this section we will prove Corollaries \ref{Cor: 2NCID} and \ref{Cor: oddNCIDs} and then construct a number of examples. For our proof, we will combine the main theorem of \cite{Pleasants} with Theorems \ref{Thm: Main} and \ref{Thm: MainOddp}. 

Recall that a common $N$-index divisor (C$N$-ID) is a prime that divides the index of any order with $N$ ring generators in the maximal order. Also, recall that the number of monic irreducible polynomials of degree $f$ over $\FF_{q}$ is
\[\Irr(f,q) = \frac{1}{f}\sum_{d\mid f}\mu\left(\frac{f}{d}\right)q^d.\]
Finally, given an integral prime $p$ and a number field $K/\QQ$, let $\lambda_f$ be the number of primes of residue class degree $f$ above $p$ in $\Ocal_K$. 
Pleasants's main theorem in the case where the base field is $\QQ$ follows:

\begin{theorem}[Theorem 3 of \cite{Pleasants}]\label{Thm: PleasantsMain}
    The prime $p$ is a C$N$-ID of $K/\QQ$ if and only if $\lambda_f>\Irr\big(f,p^N\big)$ for some positive integer $f$.
\end{theorem}

Now we proceed with the proof of Corollary \ref{Cor: 2NCID}.

\begin{proof}[Proof of Corollary \ref{Cor: 2NCID}.]
    First, it is useful to note that ramification indices are not relevant in Theorem \ref{Thm: PleasantsMain}; we need only focus on the number of primes of a given residue class degree. Now, Cases \hyperref[CorNCIDsIandII]{I \& II} are established by noting that, aside from ramification indices, the splitting of $2\Ocal_{\QQ(\hspace{-.5ex}\sqrt[n]{a})}$ mirrors the splitting of a polynomial into irreducibles in $\FF_2[x]$. In other words, $\lambda_f\leq \Irr\big(f,2\big)$ since $\lambda_f$ is equal to the number of irreducibles of degree $f$ in the factorization of a separable polynomial in $\FF_2[x]$.

    Moving to Case \ref{CorNCIDsIII}, we see that the same logic shows that 2 is not a C$N$-ID when $w=1$ since the splitting of $2\Ocal_{\QQ(\hspace{-.5ex}\sqrt[n]{a})}$ mirrors that of $x^{n_0}-1$ in $\FF_2[x]$. Thus we have Case \ref{CorNCIDsIIIi}. For Case \ref{CorNCIDsIIIii}, we take an $f>0$ and note there are $D_f+(w-2)d_f$ primes of degree $f$ above $2$ in $\Ocal_{\QQ(\hspace{-.5ex}\sqrt[n]{a})}$. Theorem \ref{Thm: PleasantsMain} finishes the argument. Likewise, in Case \ref{CorNCIDsIIIiii} we have $(m+1)d_f$ primes of degree $f$ above $2$, so Theorem \ref{Thm: PleasantsMain} again completes the argument. Since splitting and inertia are the same in Theorem \ref{Thm: Main} Case \ref{MainIV0} but with $x^{\gcd(h,n_0)}-1$ replacing $x^{n_0}-1$ and $w_0$ replacing $w$, the argument above also yields Case \ref{CorNCIDsIV0} of Corollary \ref{Cor: 2NCID}.

    We proceed with Case \ref{CorNCIDsIVi}. In Case \ref{MainIVi1} of Theorem \ref{Thm: Main} we see 2 is not a C$N$-ID since $\lambda_f$ is equal to the number of irreducibles of degree $f$ in the factorization of $x^{\gcd(h,n_0)}-1$ in $\FF_2[x]$. Thus, in order for 2 to be a C$N$-ID, it is necessary that $v_2(a_0+1)\geq 3$ and $k=1$. When this occurs, the number of primes of degree $f$ above $2$ is $D_f$ if $v_2(a_0+1)= 3$ and $2d_f$ if $v_2(a_0+1)\geq 4$. Thus, with our necessary hypotheses, 2 is a C$N$-ID if and only if $v_2(a_0+1)= 3$ and $D_f>\Irr\big(f,2^N\big)$ or $v_2(a_0+1)\geq 4$ and $2d_f>\Irr\big(f,2^N\big)$. 

    We now want to argue that in Case \ref{CorNCIDsIVi}, we must have $N=1$. When $v_2(a_0+1)= 3$, the factorization is the same as in Case \ref{CorNCIDsIVii}, cf. Theorem \ref{Thm: Main} Cases \ref{MainIVi2} and \ref{MainIVii}, and we will make the requisite argument when we address Case \ref{CorNCIDsIVii} below. Thus, we suppose $v_2(a_0+1)\geq 4$ and $2d_f>\Irr\big(f,2^N\big)$. For a contradiction, assume that $N>1$. We see that in this situation $f>2$. Indeed, for $f=1$, we see that $x-1$ is the only degree 1 factor of $x^{\gcd(h,n_0)}-1$, and $2\not>\Irr\big(1,2^N\big)$ for any $N>0$. For $f=2$, there is only one irreducible of degree 2 in $\FF_2[x]$, so the following string of inequalities shows irreducibles of degree 2 will not cause $2$ to be a C$N$-ID:
    \[2d_f\leq 2\leq 2^{N-1}\leq 2^{N-1}\left( 2^N-1\right) =\frac12\left(2^{2N}-2^N\right)=\Irr\big(f,2^N\big).\]

    We continue with the assumptions that $f>2$ and $N>1$ in search of a contradiction. Following Pleasants \cite[Equation (2)]{Pleasants}, we will employ the following bounds on $\Irr\big(f,q\big)$:
    \begin{equation}\label{Eq: Pleasants(2)}
        \frac{q^{f-1}(q-1)}{f}\leq \Irr\big(f,q\big) \leq \frac{q^f}{f}.
    \end{equation}
    These are established via an analysis of Gauss's formula. If $2$ is a C$N$-ID, we have $2d_f>2^{Nf-N}\big(2^N-1\big)/f$. Dividing by $2$ and using the fact that $d_f\leq 2^f/f$, we have $2^f/f>2^{Nf-N-1}\big(2^N-1\big)/f$. Thus $2^f>2^{Nf-N-1}\big(2^N-1\big)$. We have a contradiction if we can show that $f\leq Nf-N-1$ or equivalently $N\geq \frac{f+1}{f-1}$. This holds, however, since $f>2$ and $N>1$. Thus we conclude that in Case \ref{CorNCIDsIVi} it is necessary for $N=1$ in order for $2$ to be a C$N$-ID.

    For Case \ref{CorNCIDsIVii} and the subcase of \ref{CorNCIDsIVi} with $v_2(a_0+1)=3$, we see $\lambda_f$ is equal to the number of irreducible factors of degree $\frac{f}{2}$ in the factorization of $x^{\gcd(h,n_0)}-1$ over $\FF_4[x]$. Here $f$ is necessarily even and $2$ is a C$N$-ID if and only if $D_f>\Irr\big(f,2^N\big)$. Note that though the splitting of $2$ mirrors the splitting of a separable polynomial, common index divisors are a possibility here since the splitting of $x^{\gcd(h,n_0)}-1$ is over $\FF_4[x]$ as opposed to $\FF_2[x]$. For example, if $\gcd(h,n_0)=3$, then 2 is a common 1-index divisor. It remains to show that $N=1$ is the only possibility in this case. For a contradiction, we suppose $N\geq 2$. We will employ \eqref{Eq: Pleasants(2)}. First, we contend with the $f=2$ case. Since $x^{\gcd(h,n_0)}-1$ can have at most three linear factors over $\FF_4[x]$, it suffices to note that $\Irr\big(2,2^N\big)\geq 2^{2N-N}\big(2^N-1\big)/2\geq 3$. Now suppose $f>2$. Note that $\lambda_f\leq 2^{f+1}/f$. 
    Likewise, $\Irr\big(f,2^N\big)\geq 2^{Nf-N}\big(2^N-1\big)/f$. 
    We have $2^{f+1}/f>2^{Nf-N}\big(2^N-1\big)/f$. For a contradiction it suffices to show $f+1\leq Nf-N$. As $f>2$, we have $\frac{f+1}{f-1}\leq 2\leq N$. This establishes that in this case, if 2 is a C$N$-ID, then $N=1$.
    
    We continue on to Cases \ref{CorNCIDsIViii1new} and \ref{CorNCIDsIViii3new}. Theorem \ref{Thm: Main} Cases \ref{MainIViii1} and \ref{MainIViiinew3} show there are $D_f+(w_0-2)d_f$ primes above 2 of degree $f$. Applying Theorem \ref{Thm: PleasantsMain}, we have the result.

    For Case \ref{CorNCIDsIViii2new}, we see that Theorems \ref{Thm: Main} and \ref{Thm: PleasantsMain} show that 2 is a C$N$-ID if and only if $2d_f>\Irr\big(f,2^N\big)$. Note that $f>1$ since $x$ is not a factor of $x^{\gcd(h,n_0)}-1$ and $2\leq 2^N$. Since $f>1$, the argument employed for Case \ref{CorNCIDsIVi} shows that $N=1$. Thus we have our result in this case.

    Finally, for Case \ref{CorNCIDsIViii4}, Theorem \ref{Thm: Main} shows the number of primes above 2 of residue class degree $f$ is $(k+2)d_f$. Applying Theorem \ref{Thm: PleasantsMain} finishes our proof.
\end{proof}


We now undertake the proof of Corollary \ref{Cor: oddNCIDs}. 

\begin{proof}[Proof of Corollary \ref{Cor: oddNCIDs}.]
    First we note that the case when $N=1$ is Corollary 1.8 of \cite{SmithRadicalSplitting}. Thus, Cases \hyperref[CorOddNCIDsIandII]{I \& II} are covered by \cite[Corollary 1.8]{SmithRadicalSplitting} since a C$N$-ID with $N>1$ is necessarily a common 1-index divisor.

    Moving to Case \ref{CorOddNCIDsIII}, Theorem \ref{Thm: MainOddp} shows we have $\min(w,m+1)\cdot d_f$ primes of degree $f$ above $p$. With this, Theorem \ref{Thm: PleasantsMain} completes the argument.

    Finally, in Case \ref{CorOddNCIDsIV}, Theorem \ref{Thm: MainOddp} shows we have $\mathfrak{d}_{f}+\min(w_0-1,k,m)\cdot d_f$ primes of degree $f$ above $p$. Again, Theorem \ref{Thm: PleasantsMain} yields our result.
\end{proof}


\begin{example}\label{Ex: Using2togetNgenerators}
    Fix an integer $N>1$, and suppose we wish to construct a radical extension needing at least $N+1$ ring generators. In other words, we want to construct an extension with a common $N$-index divisor. Theorems \ref{Thm: Main} and \ref{Thm: MainOddp} give us a number of ways to obtain such a construction. For a simple example, we will use Case \ref{MainIIIiii} of Theorem \ref{Thm: Main}.

    We take $n=2^{2^{N}}$ and $a=2^{2^{N}+2}+1$. The Vahlen-Capelli theorem shows $x^n-a$ is irreducible. We see $w=2^N+2$, so $w>m+1=2^N+1$. Thus 
    \[2\Ocal_{\QQ(\hspace{-.5ex}\sqrt[n]{a})}=\prod_{i=1}^{m+1}\gp_{i}^{2^{\max(i-2,0)}} =\gp_1\gp_2\prod_{i=3}^{2^N+1}\gp_{i}^{2^{i-2}}.\]
    We see that $2$ splits into $2^N+1$ degree 1 primes. However, there are only $2^N$ degree 1 polynomials in $\FF_{2^N}[x]$. Thus 2 is a common $N$-index divisor, and $\Ocal_{\QQ(\hspace{-.5ex}\sqrt[n]{a})}$ requires at least $N+1$ ring generators. In fact, it requires exactly $N+1$ ring generators since for a prime to be a common 1-index divisor in a radical extension, it is necessary that the prime divide the degree of the extension. See Corollary 1.6 of \cite{SmithRadicalSplitting}. 
\end{example}

Notice that in Example \ref{Ex: Using2togetNgenerators} the degree is quite large. From Theorem \ref{Thm: PleasantsMain}, the minimum degree required for $p$ to be a common $N$-index divisor is $p^N+1$. We can achieve this via the following construction.

\begin{proposition}\label{Prop: NonRadCommonN}
    Fix a prime $p$ and a positive integer $N$. Letting $\ell$ be any prime not equal to $p$, the extension given by
    \[f(x)=x^{p^N+1}+\ell p x^{p^N}+\ell p^3x^{p^N-1}+\cdots + \ell p^{\frac{i(i+1)}{2}}x^{p^N-i+1}+\cdots \ell p^{p^N(p^N+1)/2}x+\ell p^{(p^N+1)(p^N+2)/2}\] has $p$ as a common $N$-index divisor. In other words, the ring of integers of $\QQ[x]/\langle f(x) \rangle$ requires at least $N+1$ ring generators.
\end{proposition}

\begin{proof}
    Since the polynomial is $\ell$-Eisenstein, it is irreducible. We apply Theorem \ref{Thm: Ore}. The principal $x$-polygon has $p^N+1$ sides, each with length 1. From left to right, the slopes descend from $p^N+1$ to 1. Since there are only $p^N$ distinct monic linear polynomials in $\FF_{p^N}[x]$, Pleasants (Theorem \ref{Thm: PleasantsMain}) tells us that at least $N+1$ ring generators are required.
\end{proof}


\begin{example}\label{Ex. FourGens}
    To illustrate Proposition \ref{Prop: NonRadCommonN}, consider \begin{equation}\label{Eq. fdeg9}
        f(x)=x^9+6x^8+24x^7+2^6\cdot 3x^6 + 2^{10}\cdot3x^5 + 2^{15}\cdot3x^4 + 2^{21}\cdot3x^3 + 2^{28}\cdot3x^2 + 2^{36}\cdot3x + 2^{45}\cdot 3.\end{equation}
    The polynomial is 3-Eisenstein, hence irreducible. We let $K$ be the number field generated by a root, and we write $\Ocal_K$ for the ring of integers. We begin to apply Theorem \ref{Thm: Ore}. Reducing at 2, we take the $x$-adic development which is simply \eqref{Eq. fdeg9}. The principal $x$-polygon is given in Figure \ref{Fig. fdeg9}. It has 9 sides, each with length 1, so $2\Ocal_K$ factors into 9 distinct degree 1 primes. As there are only 8 distinct irreducibles of degree 1 in $\FF_{2^3}[x]$, we see that 2 is a common $3$-index divisor. In other words, $\Ocal_K$ cannot be written $\ZZ[\alpha_1, \alpha_2, \alpha_3]$ for algebraic integers $\alpha_1, \alpha_2,\alpha_3$. Notice that this example is minimal in the sense that for number fields of degree 8 and smaller common 3-index divisors do not occur.    
\end{example}

\begin{figure}[h!]
\begin{tikzpicture}
      \draw[<->] (0,10.5) -- (0,0) -- (9.5,0);
      
      \draw [fill] (0,45*2/9) circle [radius = .05];
	   \draw [fill] (1,36*2/9) circle [radius = .05]; 
      \draw [fill] (2,28*2/9) circle [radius = .05];
      \draw [fill] (3,21*2/9) circle [radius = .05];
      \draw [fill] (4,15*2/9) circle [radius = .05]; 
      \draw [fill] (5,10*2/9) circle [radius = .05];
      \draw [fill] (6,6*2/9) circle [radius = .05];
      \draw [fill] (7,3*2/9) circle [radius = .05]; 
      \draw [fill] (8,1*2/9) circle [radius = .05];
      \draw [fill] (9,0) circle [radius = .05];
      \foreach \x in  {1,2,3,4,5,6,7,8,9} {
      		\draw (\x, 2pt) -- +(0,-4pt);
            \node [below] at (\x, 0) {\x};
      }
      \foreach \y in  {5*2/9,10*2/9,15*2/9,20*2/9, 25*2/9, 30*2/9, 35*2/9, 40*2/9, 10} {
      		\draw (2pt, \y) -- +(-4pt, 0);
      }
      \node [left] at (0,5*2/9) {5};    
      \node [left] at (0,10*2/9) {10};
      \node [left] at (0,15*2/9) {15};    
      \node [left] at (0,20*2/9) {20};
      \node [left] at (0,25*2/9) {25};    
      \node [left] at (0,30*2/9) {30};
      \node [left] at (0,35*2/9) {35};    
      \node [left] at (0,40*2/9) {40};
      \node [left] at (0,45*2/9) {45};
      
      \draw (0,10) -- (1,36*1*2/9);
      \draw (1,36*2/9) -- (2,28*1*2/9);
      \draw (2,28*1*2/9) -- (3,21*1*2/9);
      \draw (3,21*1*2/9) -- (4,15*1*2/9);
      \draw (4,15*2/9) -- (5,10*1*2/9);
      \draw (5,10*2/9) -- (6,6*1*2/9);
      \draw (6,6*1*2/9) -- (7,3*1*2/9);
      \draw (7,3*1*2/9) -- (8,1*2/9);
      \draw (8,2/9) -- (9,0);
	\end{tikzpicture}
  \caption{The 9-sided principal $x$-polygon for Example \ref{Ex. FourGens}}
  \label{Fig. fdeg9}
\end{figure}


\section{Examples of the Main Theorem}\label{Sec: Examples} 

We will illustrate Theorem \ref{Thm: Main} via a series of examples, varying $n$ and $a$ to illuminate different cases. Of course, Case \ref{MainI} is simply Dedekind--Kummer factorization. For example, take $n=a=3$, though $a$ could be any odd integer. We see \[x^3-3\equiv x^3-1\equiv (x-1)(x^2+x+1) \bmod 2.\text{ Thus } 2\Ocal_{\QQ(\sqrt[3]{3})}=\big(2,\sqrt[3]{3}-1\big)\big(2, \sqrt[3]{9}+\sqrt[3]{3}+1\big).\] 

\begin{example}\label{Ex. n=48.vary a}
    Take $n=48=2^4\cdot 3$ and $a=6$. Note $\gcd\big(v_2(a),n\big)=1$, so we are in Case \ref{MainII}.
    We have 
    \begin{align*}
        g(x)=x-1, \text{ so } 2\Ocal_{\QQ(\hspace{-.5ex}\sqrt[48]{6})}=\gp^{48}.
    \end{align*}
    Here we see that $x^{48}-6$ is actually $2$-Eisenstein, so Dedekind--Kummer factorization still yields the result. Moreover, $2\Ocal_{\QQ(\hspace{-.5ex}\sqrt[48]{6})}=\big(2,\sqrt[48]{6}\big)^{48}$. 

    When $\gcd(v_2(a),n)=1$ in Case \ref{MainII}, we can produce explicit generators of the prime above 2. For example, if $a=96=2^5\cdot 3$, then we note that $29\cdot 5-3\cdot 48=1$. Hence, $\beta=\big( \sqrt[48]{96}\big)^{29} /2^3$ is an integral element of the correct 2-adic valuation. We see $2\Ocal_{\QQ(\hspace{-.5ex}\sqrt[48]{96})}=(2,\beta)^{48}$.

    Even when $v_2(a)$ and $n$ are not relatively prime, we can construct generators. However, the process becomes somewhat ad hoc. Take $a=24=2^3\cdot 3$. Case \ref{MainII} yields $2\Ocal_{\QQ(\hspace{-.5ex}\sqrt[48]{24})}=\gp_1^{16}\gp_2^{16}$ with the residue class degree of $\gp_2=2$. However, we see $\big(\hspace{-.5ex}\sqrt[48]{24}\big)^{16}=2\cdot \sqrt[3]{3}$, so our example above with $\sqrt[3]{3}$ is contained in this extension. Since $\sqrt[48]{24}$ has the requisite valuation ($\frac{1}{16}$ that of 2), we can argue that 
    \[\gp_1=(2, \sqrt[3]{3}-1, \sqrt[48]{24}\big) \ \text{ and } \gp_2=\big(2, \sqrt[3]{9}+\sqrt[3]{3}+1, \sqrt[48]{24}\big).\] 
    This construction of generators is too unwieldy to employ in a generic manner.
    
    We continue with $n=48$ but take $a=3$. Since $a$ is odd, this is Case \ref{MainIII}. As $a-1=2$, we have $w=1$ and we are in Case \ref{MainIIIi}. We see $g(x)=x^3-1=(x-1)(x^2+x+1)$ in $\FF_2[x]$, so $2\Ocal_{\QQ(\hspace{-.5ex} \sqrt[48]{3})}=\gp_1^{16}\gp_2^{16}$, where the residue class degrees are 1 and 2.

    Taking $a=5$ now, notice $w=2$, so we are in Case \ref{MainIIIii}. Since $g(x)=x^3-1=(x-1)(x-\zeta_3)(x-\zeta_3^2)$ in $\FF_4[x]$, we have $2\Ocal_{\QQ(\hspace{-.5ex} \sqrt[48]{5})}=\gp_1^8\gp_2^8\gp_3^8,$ where each prime has residue class degree 2.

    Taking $a=41$, $w=3$, so we are again in Case \ref{MainIIIii}. Since $g(x)=x^3-1=(x-1)(x-\zeta_3)(x-\zeta_3^2)$ in $\FF_4[x]$ and $g(x)=(x-1)(x^2+x+1)$ in $\FF_2[x]$, we have $2\Ocal_{\QQ(\hspace{-.5ex} \sqrt[48]{41})}=\gp_1^4\gp_2^4\gp_3^4 \gp_{2,1}^8\gp_{2,2}^8,$ where $\gp_{2,1}$ has residue class degree 1, but every other prime ideal factor has residue class degree 2.
    
    If $a=65$, then $w=6$. Since $m+1=5$, we are in Case \ref{MainIIIiii}. We employ the aforementioned factorization of $g(x)=x^3-1$ over $\FF_2[x]$, to obtain
\[2\Ocal_{\QQ(\hspace{-.5ex} \sqrt[48]{65})}=\gp_{1,1}\gp_{1,2}\gp_{2,1}\gp_{2,2}\gp_{3,1}^2\gp_{3,2}^2\gp_{4,1}^4\gp_{4,2}^4\gp_{5,1}^8\gp_{5,2}^8,\]
where the second index also gives the residue class degree.

    Now take $a=12$. Here $\gcd\big(48, v_2(12)\big)=2$, so we are in Case \ref{MainIV}. As $a=a_02^{h2^k}$, we find $a_0=3$, $h=1$, and $k=1$. Since $h=1$, we have $g(x)=x-1$. As $n=2^mn_0=2^43$, we see $k<m$, so we avoid Case \ref{MainIV0}. Further $w_0=1$, so we are in Case \ref{MainIVi}. We have $v_2(a_0+1)=2$, so the splitting is provided by \ref{MainIVi1}. We find $2\Ocal_{\QQ(\hspace{-.5ex} \sqrt[48]{12})}=\gp^{48}$.
    
    Letting $a=448=2^6\cdot 7$, we have $h=3$ and $k=1$, so $d=\gcd(h,n_0)=3$ and $g(x)=x^3-1$. We still have $w_0=v_2(7-1)=1$, but since $v_2(a_0+1)=3$ we consult Case \ref{MainIVi2}. As $x^3-1$ splits completely over $\FF_4[x]$, we see $2\Ocal_{\QQ(\hspace{-.5ex} \sqrt[48]{448})}=\gp_1^{8}\gp_2^8\gp_3^8$, where the prime ideals have residue class degree 2.
    
    Taking $a=1984=2^6\cdot 31$, we have much the same situation as above but $v_2(31+1)=5$, so we are in Case \ref{MainIVi3}. Since $x^3-1=(x-1)(x^2+x+1)$ over $\FF_2[x]$, we have $2\Ocal_{\QQ(\hspace{-.5ex} \sqrt[48]{1984})}=\gp_{1,1}^{8}\gp_{1,2}^{8}\gp_{2,1}^{8}\gp_{2,2}^{8}$, where the second index can be taken to be the degree.

    Continuing, we try $a=320=5\cdot 2^6$. We have $w_0=v_2(a_0-1)=2$, so Case \ref{MainIVii} and the complete splitting of $x^3-1$ over $\FF_4[x]$ gives us $2\Ocal_{\QQ(\hspace{-.5ex} \sqrt[48]{320})}=\gp_1^{8}\gp_2^8\gp_3^8$, with each $\gp_i$ having residue field $\FF_4$.
\end{example}


\begin{example}\label{Ex. n=80}
    For some variation, we take $n=80$ in order to explore Cases \ref{MainIV0} and \ref{MainIViii}. We see $m=4$ and $n_0=5$. Take $a=2^{5\cdot 32}\cdot3$. We see that $5=k\geq m$ and $w_0=1$, so we are in Case \ref{MainIV0i}. As $d=5$, we compute $g(x)=x^5-1=(x-1)(x^4+x^3+x^2+x+1)$ in $\FF_2[x]$. Thus $2\Ocal_{\QQ(\hspace{-.5ex} \sqrt[80]{a})}=\gp_1^{16}\gp_2^{16},$ where $\gp_1$ has degree 1 but $\gp_2$ has degree 4.

    Notice here that $\QQ\big(\hspace{-.5ex}\sqrt[16]{a}\big)=\QQ\big(\hspace{-.5ex}\sqrt[16]{3}\big)$, but $x^5-a$ yields a different extension of this intermediate field than $x^5-3$. This is to illustrate that Case \ref{MainIV0} can be simplified to Case \ref{MainIII} at the 2-power level, but afterwards the behavior is distinct.

    Now take $a=42,991,616=2^{4\cdot 5}\cdot 41$. We see that $2\mid a$ and $2\mid \gcd(v_2(a),n)$, so we are in Case \ref{MainIV}. As $w_0=v_2(40)=3$ and $k=2$, we are in Case \ref{MainIViii1}. Since $d=5$, we compute the factorization of $g(x)=x^5-1$ over $\FF_2[x]$ and $\FF_4[x]$. The factorization over $\FF_2[x]$ is described above, and over $\FF_4[x]$ we have $x^5-1=(x-1)(x^2 + \zeta_3 x + 1) (x^2 + (\zeta_3 + 1)x + 1)$. Thus $2\Ocal_{\QQ(\hspace{-.5ex} \sqrt[80]{a})}=\gp_{0,1}^{4}\gp_{0,2}^{4}\gp_{0,3}^{4} \gp_{1,1}^{8}\gp_{1,2}^{8}$, where $\gp_{0,1}$ has degree 2, $\gp_{0,2}, \gp_{0,3}, \gp_{1,2}$ have degree 4, and $\gp_{1,1}$ has degree 1. 

    Continuing, take $a=17,825,792=2^{4\cdot 5}\cdot 17$. We are in Case \ref{MainIV} with $w_0=4$ and $k=2$. Since $w_0-2 = k >1$, we are in Case \ref{MainIViiinew3}. The factorizations of $g(x)$ above yield $2\Ocal_{\QQ(\hspace{-.5ex} \sqrt[80]{a})}=\gp_{2,1}^{4}\gp_{2,2}^{4}\gp_{2,3}^{4}\gp_{0,1}^{4}\gp_{0,2}^{4}\gp_{1,1}^{4}\gp_{1,2}^{4},$ where $\gp_{2,1}$ has residue class degree 2, each $\gp_{*,2}$ and $\gp_{2,3}$ has residue class degree 4, and $\gp_{1,1}$ and $\gp_{0,1}$ have residue class degree 1. 
   
    Now take $a=34,603,008=2^{4\cdot 5}\cdot 33$. We are in Case \ref{MainIV} with $w_0=5$ and $k=2$. Since $w_0-2 > k >1$, we are in Case \ref{MainIViiinew4}. As above, we have $g(x)=x^5-1=(x-1)(x^4+x^3+x^2+x+1)$ in $\FF_2[x]$. Thus $2\Ocal_{\QQ(\hspace{-.5ex} \sqrt[80]{a})}=\gp_{0,1}^{4}\gp_{0,2}^{4}\gp_{1,1}^{4}\gp_{1,2}^{4}\gp_{2,1}^{4}\gp_{2,2}^{4}\big(\gp_{2,1}'\big)^{4}\big(\gp_{2,2}'\big)^{4},$ where 
    each $\gp_{*,1}$ and $\gp'_{1,1}$ have degree 1 and each $\gp_{*,2}$ and $\gp'_{2,2}$ have degree 4. 
    
\end{example}


\begin{example}\label{Ex. CaseIVii3} 
    Finally, we will look at an example of Case \ref{MainIViiinew2} with $m=2$. Take $n=2^2\cdot 21$. We need $k=1$, otherwise we end up in Case \ref{MainIV0}, so let $a=65\cdot 2^{2\cdot 21}$. We see $w_0=v_2(65-1)=6$ and $t=w_0-2=4$. Further, $d=\gcd(h,n_0)=21$, so $g(x)=x^{21}-1=(x - 1)(x^2 + x + 1)(x^3 + x + 1)(x^3 + x^2 + 1)(x^6 + x^4 + x^2 + x + 1)(x^6 + x^5 + x^4 + x^2 + 1)$ in $\FF_2[x]$. Case \ref{MainIViiinew2} shows $2\Ocal_{\QQ(\hspace{-.5ex} \sqrt[84]{a})}=\gp_{1,1}^{2}\gp_{1,2}^{2}\gp_{1,3}^{2}\gp_{1,4}^{2}\gp_{1,5}^{2}\gp_{1,6}^{2}\gp_{2,1}^{2}\gp_{2,2}^{2}\gp_{2,3}^{2}\gp_{2,4}^{2}\gp_{2,5}^{2}\gp_{2,6}^{2},$ where each $\gp_{*,1}$ has degree 1, each $\gp_{*,2}$ has degree 2, each $\gp_{*,3}$ and $\gp_{*,4}$ has degree 3, and each $\gp_{*,5}$ and $\gp_{*,6}$ has degree 6. 
\end{example}


\section{Non-Monogenic Fields with No Common Index Divisors}\label{Sec: NonMonoNoCIDs}


Following Pleasants, call a number field $K$ \textit{exceptional}\footnote{Recent work of Alp\"oge, Bhargava, and Shnidman \cite{CubicNonMono} shows that these number fields might not be so rare.} if it has no common index divisors yet it is still non-monogenic. 
When $n+1$ is prime, Pleasants \cite{Pleasants} uses congruence conditions on discriminants to construct radical (pure) number fields of degree $n$ that have no common index divisors yet are non-monogenic. 
We will employ the index form to construct other exceptional radical extensions. This construction yields exceptional examples of any degree $n>2$. 

\begin{proposition}\label{Prop: Exceptional}
    Fix $n>2$. Let $p$ be a prime and $k>1$ a squarefree integer such that 
    \begin{enumerate}[label=(\arabic*), ref=(\arabic*)]
        \item $p\dnd nk$;\label{(1)}
        \item $\big(pk^{n-1}\big)^q\not\equiv  p k^{n-1} \bmod q^2$ or $\big(p^{n-1}k\big)^q\not\equiv  p^{n-1}k \bmod q^2$ for any prime $q\mid n$;\label{(2)}
        \item $k^{(n-1)(n-2)/2}X_1^{n(n-1)/2}\equiv \pm 1\bmod p$ has no solutions.\label{(3)}
    \end{enumerate}
    Then, $\QQ\big(\sqrt[n]{pk^{n-1}}\big)$ is an exceptional number field. Moreover, there are infinitely many such fields for each $n$.
\end{proposition}

\begin{proof}
    First, we will justify the fact that there are infinitely many choices that satisfy the conditions for each $n>2$. What follows is merely one possible construction, and it seems likely that a more careful argument could establish a positive density. Choose $p$ to be a prime such that $p\equiv 1\bmod n^2$. Choose $k$ to be a prime such that $k\equiv 1+q\bmod q^2$ for each prime $q\mid n$ and $k\equiv \xi\bmod p$, where $\xi$ generates $\ZZ/p\ZZ^\times$. The Chinese (Sunzi's) remainder theorem and Dirichlet's prime number theorem ensure that there are infinitely many such choices. We see \ref{(1)} and \ref{(2)} are satisfied.

    We focus on $k^{(n-1)(n-2)/2}X_1^{n(n-1)/2}\equiv \pm 1\bmod p$. A solution implies, $X_1^{n(n-1)}\equiv k^{-(n-1)(n-2)}\bmod p$. Since $n\mid \gcd(p-1,n)$, we see $X_1^{n(n-1)}\bmod p$ is an $n^{\text{th}}$ power. As $k\equiv \xi\bmod p$ and $n>2$, we find $k^{(n-1)(n-2)}\bmod p$ is not an $n^\text{th}$ power. This contradiction implies \ref{(3)} is satisfied. 

    Now, we prove that $\QQ\big(\sqrt[n]{pk^{n-1}}\big)$ is not monogenic but has no common index divisors. Condition \ref{(2)} ensures that $x^n - pk^{n-1}$ or $x^n - p^{n-1}k$ is maximal at each prime dividing $n$. (See \cite{g17}, \cite{JhorarKhanduja}, or \cite{SmithRadical}.) The discriminant of $x^n - pk^{n-1}$ is divisible by only $p$ and the primes dividing $nk$. Thus, employing Condition \ref{(1)} to note $x^n - pk^{n-1}$ is $p$-Eisenstein while $x^n - p^{n-1}k$ is Eisenstein at all primes dividing $k$, we see that 
\[\Bcal = \left\{1, \sqrt[n]{p k^{n-1}}=p^{\frac1n}k^{\frac{n-1}{n}}, p^{\frac2n}k^{\frac{n-2}{n}}, \dots, p^{\frac{n-1}{n}} k^{\frac{1}{n}}  \right\}\]
is an integral basis for $\QQ\big(p^\frac{1}{n}k^\frac{n-1}{n}\big)$. 

If $\QQ\big(p^\frac{1}{n}k^\frac{n-1}{n}\big)$ is monogenic with monogenerator $\alpha=X_1 p^{\frac1n}k^{\frac{n-1}{n}} + X_2 p^{\frac2n}k^{\frac{n-2}{n}}+\cdots + X_{n-1}p^{\frac{n-1}{n}}k^{\frac{1}{n}}$ (we lose no generality assuming $X_0=0$), then the following change of basis matrix has determinant $\pm 1$:

\begin{equation}\label{Eq: IndMatrix}
\text{Index Form}\left(\QQ\left(p^\frac{1}{n}k^\frac{n-1}{n}\right)\right) \ \  = \ \ \det\begin{bmatrix}  
    1     		 & 0            & 0       & \dots      & 0       & 0 \\
    0      	  	 & X_1          & X_2     & \dots      & X_{n-2} & X_{n-1} \\
    A_{3,1}	 	 & A_{3,2}	    & A_{3,3} & \dots     & A_{3,n-1}  & A_{3,n} \\
    A_{4,1}      & A_{4,2}      & A_{4,3} & \dots    & A_{4,n-1}  & A_{4,n} \\  
    \vdots       & \vdots       & \vdots  & \ddots    & \vdots  & \vdots \\
    A_{n-1,1}      & A_{n-1,2}      & A_{n-1,3} & \dots    & A_{n-1,n-1}  & A_{n-1,n} \\  
    A_{n,1}      & A_{n,2}      & A_{n,3} & \dots    & A_{n,n-1}  & A_{n,n} \\  
\end{bmatrix}.
\end{equation}

Here the entry $A_{i,j}$ is the coefficient of the $j^{\text{th}}$ element of $\Bcal$ when we compute 
\begin{equation}\label{Eq. Aij}
\left(X_1 p^{\frac1n}k^{\frac{n-1}{n}} + X_2 p^{\frac2n}k^{\frac{n-2}{n}}+\cdots + X_{n-1}p^{\frac{n-1}{n}}k^{\frac{1}{n}}\right)^{i-1}.
\end{equation}

We will compute the determinant in \eqref{Eq: IndMatrix} modulo $p$. Expanding across the first row, we see it suffices to compute the determinant of 
\[M\coloneqq\begin{bmatrix}  
X_1           & X_2     & X_3     & \dots      & X_{n-2} & X_{n-1} \\
A_{3,2}	      & A_{3,3} & A_{3,4} & \dots     & A_{3,n-1}  & A_{3,n} \\
A_{4,2}       & A_{4,3} & A_{4,4} & \dots    & A_{4,n-1}  & A_{4,n} \\  
\vdots        & \vdots  & \vdots    & \ddots &\vdots  & \vdots \\
A_{n-1,2}       & A_{n-1,3} & A_{n-1,4} &\dots    & A_{n-1,n-1}  & A_{n-1,n} \\  
A_{n,2}       & A_{n,3} & A_{n,4} & \dots    & A_{n,n-1}  & A_{n,n} \\  
\end{bmatrix}\]

Noting that products 
with $p$-adic valuation 1 or greater result in an entry that is divisible by $p$, we see that $M$ is upper triangular modulo $p$. In other words, if $i>j$, the entry $A_{i,j}$ will be divisible by $p$. Indeed, in this case, \eqref{Eq. Aij} shows each summand making up $A_{i,j}$ is coming from a product of basis elements resulting in a $p$-adic valuation of 1 or greater. Thus, the determinant of $M$ modulo $p$ is $X_1A_{3,3}\cdots A_{n,n}$. Again, considering \eqref{Eq. Aij} for $A_{i,i}$ and eliminating summands that are divisible by $p$ shows that we have
\[\det M\equiv\det\begin{bmatrix}    
 X_1           & X_2     & X_3 & \dots      & X_{n-2} & X_{n-1} \\
0	      & kX_1^2 &  A_{3,4} & \dots     & A_{3,n-1}  & A_{3,n} \\
0       & 0 &  k^2X_1^3 & \cdots & A_{4,n-1} & A_{4,n} \\  
\vdots        & \vdots & \vdots & \ddots    & \vdots  & \vdots \\
0       & 0 & 0 &\dots    & k^{n-3}X_1^{n-2}  & A_{n-1,n} \\  
0       & 0 & 0 & \dots    & 0  & k^{n-2}X_1^{n-1} \\  
\end{bmatrix}
\equiv k^{(n-1)(n-2)/2}X_1^{n(n-1)/2}\bmod p.\]

Condition \ref{(3)} shows there are no solutions.

To see there are no common index divisors, we can employ Corollaries \ref{Cor: 2NCID} and \ref{Cor: oddNCIDs}, or we can simply note that we have found polynomials that are maximal at each prime dividing the discriminant of $\QQ\big(p^\frac{1}{n}k^\frac{n-1}{n}\big)$.
\end{proof}


\begin{example}\label{Ex. DiscCongn=5}
    For an explicit example of the construction above, take $x^5-11\cdot 13^4$. In the notation above $n=5$, $p=11$, and $k=13$. We check that $\big(11\cdot 13^4\big)^5\equiv 1\bmod 25$, but $11\cdot 13^4\equiv 21\bmod 25$. Further, $\gcd(11-1,5)=5>2$ and 13 reduces to 2 modulo 11, which is a generator (equivalently, the inverse of a generator) of $\ZZ/11\ZZ^\times$.

    The index form of $\QQ\big(\hspace{-.5ex}\sqrt[5]{11\cdot 13^4}\big)$ is composed of 58 monomials and takes more than a dozen lines to write down. However, modulo 11 it is simply $-2X_1^{10}$. Immediately, we see that $-2X_1^{10}=\pm 1$ has no solutions in $\ZZ/11\ZZ$, so that $\QQ\big(\hspace{-.5ex}\sqrt[5]{11\cdot 13^4}\big)$ is not monogenic. The fact that this number field has no common index divisors can be seen by noting that $x^5-11\cdot 13^4$ is maximal at all primes not equal to 13, while $x^5-11^4\cdot 13$ is maximal at 13.
\end{example}






\bibliography{Bibliography}

\begin{thebibliography}{GMN12}

\bibitem[ABS25]{CubicNonMono}
Levent Alp{\"o}ge, Manjul Bhargava, and Ari Shnidman.
\newblock A positive proportion of cubic fields are not monogenic yet have no local obstruction to being so.
\newblock {\em Mathematische Annalen}, 391:5535--5551, 2025.

\bibitem[ASW06]{MR2251710}
Saban Alaca, Blair~K. Spearman, and Kenneth~S. Williams.
\newblock Explicit decomposition of a rational prime in a cubic field.
\newblock {\em International Journal of Mathematics and Mathematical Sciences}, pages 1--11, 2006.

\bibitem[BD19]{MR3889325}
Stephen~C. Brown and Chad~T. Davis.
\newblock The factorization of 2 and 3 in cyclic quartic fields.
\newblock {\em Mathematical Journal of Okayama University}, 61:167--172, 2019.

\bibitem[Ber27]{Berwick1927}
W.E.H. Berwick.
\newblock {\em Integral Bases}.
\newblock Cambridge tracts in mathematics and mathematical physics. The University Press, 1927.

\bibitem[BL20]{BellemareLei}
Hugues Bellemare and Antonio Lei.
\newblock Explicit uniformizers for certain totally ramified extensions of the field of {$p$}-adic numbers.
\newblock {\em Abhandlungen aus dem Mathematischen Seminar der Universit\"{a}t Hamburg}, 90(1):73--83, 2020.

\bibitem[BYT24]{MR4736305}
Hamid Ben~Yakkou and Pagdame Tiebekabe.
\newblock On common index divisors and not monogenity of nonic number fields defined by trinomials of type {$x^9+ax+b$}.
\newblock {\em Bolet\'in de la Sociedad Matem\'atica Mexicana. Third Series}, 30(2):Paper No. 44, 25, 2024.

\bibitem[Ded78]{Dedekind}
Richard Dedekind.
\newblock \"{U}ber den {Z}usammenhang zwischen der {T}heorie der {I}deale und der {T}heorie der h\"{o}heren {K}ongruenzen.
\newblock {\em Abhandlungen der K\"oniglichen Gesellschaft der Wissenschaften zu G\"{o}ttingen}, 23(3):3--38, 1878.

\bibitem[EFK23]{MR4563435}
Lhoussain El~Fadil and Omar Kchit.
\newblock On index divisors and monogenity of certain septic number fields defined by {$x^7+ax ^3+b$}.
\newblock {\em Communications in Algebra}, 51(6):2349--2363, 2023.

\bibitem[FMN12]{ElFadilMontesNart}
Lhoussain~El Fadil, Jesús Montes, and Enric Nart.
\newblock Newton polygons and {$p$}-integral bases of quartic number fields.
\newblock {\em Journal of Algebra and Its Applications}, 11(04):1250073, 2012.

\bibitem[Gas17]{g17}
T.~Alden Gassert.
\newblock A note on the monogeneity of power maps.
\newblock {\em Albanian J. Math.}, 11(1):3--12, 2017.

\bibitem[GMN12]{GMN}
Jordi Gu\`ardia, Jes\'{u}s Montes, and Enric Nart.
\newblock Newton polygons of higher order in algebraic number theory.
\newblock {\em Transactions of the American Mathematical Society}, 364(1):361--416, 2012.

\bibitem[GW25]{CIDDBook}
Fernando~Q. Gouv\^{e}a and Jonathan Webster.
\newblock {\em Common inessential discriminant divisors---scenes from the early history of algebraic number theory}, volume~47 of {\em History of Mathematics}.
\newblock American Mathematical Society, Providence, RI, 2025.

\bibitem[Hen94]{Hensel1894}
K.~Hensel.
\newblock Arithmetische {U}ntersuchungen \"{u}ber die gemeinsamen ausserwesentlichen {D}iscriminantentheiler einer {G}attung.
\newblock {\em Journal f\"{u}r die Reine und Angewandte Mathematik. [Crelle's Journal]}, 113:128--160, 1894.

\bibitem[Jad09]{MR2531162}
Borka Jadrijevi\'c.
\newblock Establishing the minimal index in a parametric family of bicyclic biquadratic fields.
\newblock {\em Periodica Mathematica Hungarica. Journal of the J\'anos Bolyai Mathematical Society}, 58(2):155--180, 2009.

\bibitem[JK17]{JhorarKhanduja}
Bablesh Jhorar and Sudesh~K. Khanduja.
\newblock On the index theorem of {O}re.
\newblock {\em Manuscripta Mathematica}, 153(1-2):299--313, 2017.

\bibitem[JKK23]{MR4510645}
Anuj Jakhar, Sumandeep Kaur, and Surender Kumar.
\newblock Non-monogenity of certain octic number fields defined by trinomials.
\newblock {\em Colloquium Mathematicum}, 171(1):145--152, 2023.

\bibitem[JKS16]{JKS1}
Anuj Jakhar, Sudesh~K. Khanduja, and Neeraj Sangwan.
\newblock On prime divisors of the index of an algebraic integer.
\newblock {\em Journal of Number Theory}, 166:47--61, 2016.

\bibitem[Lan94]{LangANT}
S.~Lang.
\newblock {\em Algebraic Number Theory}.
\newblock Graduate Texts in Mathematics. Springer, 1994.

\bibitem[Lan02]{Algebra}
Serge Lang.
\newblock {\em Algebra}.
\newblock Springer, New York, NY, 2002.

\bibitem[MV76]{MannVelez}
Henry~B. Mann and William~Yslas V{\'{e}}lez.
\newblock Prime ideal decomposition in {$F(\sqrt[m]{\mu})$}.
\newblock {\em Monatshefte f\"{u}r Mathematik}, 81(2):131--139, 1976.

\bibitem[Nak83]{MR731633}
Toru Nakahara.
\newblock On the indices and integral bases of noncyclic but abelian biquadratic fields.
\newblock {\em Archiv der Mathematik}, 41(6):504--508, 1983.

\bibitem[Neu99]{Neukirch}
J\"{u}rgen Neukirch.
\newblock {\em Algebraic number theory}, volume 322 of {\em Grundlehren der Mathematischen Wissenschaften}.
\newblock Springer-Verlag, Berlin, 1999.
\newblock Translated by Norbert Schappacher.

\bibitem[Obu14]{ObusWild}
Andrew Obus.
\newblock Conductors of wild extensions of local fields, especially in mixed characteristic {$(0,2)$}.
\newblock {\em Proceedings of the American Mathematical Society}, 142(5):1485--1495, 2014.

\bibitem[Ore28a]{Ore}
{{\O}}ystein Ore.
\newblock Newtonsche {P}olygone in der {T}heorie der algebraischen {K}\"{o}rper.
\newblock {\em Mathematische Annalen}, 99(1):84--117, 1928.

\bibitem[Ore28b]{OreReview}
{{\O}}ystein Ore.
\newblock {Review: W. E. H. Berwick, Integral Bases}.
\newblock {\em Bulletin of the American Mathematical Society}, 34(3):378 -- 379, 1928.

\bibitem[Ple74]{Pleasants}
P.~A.~B. Pleasants.
\newblock The number of generators of the integers of a number field.
\newblock {\em Mathematika}, 21(2):160–167, 1974.

\bibitem[PP12]{PethoPohst}
Attila Peth{\H o} and Michael~E. Pohst.
\newblock On the indices of multiquadratic number fields.
\newblock {\em Acta Arithmetica}, 153(4):393--414, 2012.

\bibitem[Smi21]{SmithRadical}
Hanson Smith.
\newblock The monogeneity of radical extensions.
\newblock {\em Acta Arithmetica}, 198(3):313--327, 2021.

\bibitem[Smi24]{SmithRadicalSplitting}
Hanson Smith.
\newblock Prime splitting and common index divisors in radical extensions, 2024.
\newblock https://arxiv.org/abs/2409.08911.

\bibitem[SWY07]{SpearmanWilliamsYang}
Blair~K. Spearman, Kenneth~S. Williams, and Qiduan Yang.
\newblock On the common index divisors of a dihedral field of prime degree.
\newblock {\em International Journal of Mathematics and Mathematical Sciences}, 2007.

\bibitem[Tur98]{Turnwald}
Gerhard Turnwald.
\newblock Reducibility of translates of dickson polynomials.
\newblock {\em Proceedings of the American Mathematical Society}, 126(4):965--971, 1998.

\bibitem[V{\'{e}}l77]{VelezCoprime}
William~Yslas V{\'{e}}lez.
\newblock Prime ideal decomposition in {$F(\mu ^{1/m})$}. {II}.
\newblock In {\em Number theory and algebra}, pages 331--338. Academic Press, New York, 1977.

\bibitem[V{\'{e}}l78]{Velez1/p}
William~Yslas V{\'{e}}lez.
\newblock Prime ideal decomposition in {$F(\mu ^{1/p})$}.
\newblock {\em Pacific Journal of Mathematics}, 75(2):589--600, 1978.

\bibitem[V{\'{e}}l88]{VelezPrimePower}
William~Yslas V{\'{e}}lez.
\newblock The factorization of {$p$} in {${\bf Q}(a^{1/p^k})$} and the genus field of {${\bf Q}(a^{1/n})$}.
\newblock {\em Tokyo Journal of Mathematics}, 11(1):1--19, 1988.

\bibitem[Viv04]{Viviani}
Filippo Viviani.
\newblock Ramification groups and {A}rtin conductors of radical extensions of {$\Bbb Q$}.
\newblock {\em Journal de Th\'{e}orie des Nombres de Bordeaux}, 16(3):779--816, 2004.

\bibitem[WY22]{WangYuan}
Shanwen Wang and Yijun Yuan.
\newblock Uniformizer of the false {T}ate curve extension of {$\Bbb Q_p$}.
\newblock {\em Ramanujan Journal}, 58(2):549--595, 2022.

\end{thebibliography}
\bibliographystyle{alpha}


\end{document}